\documentclass[a4paper,english,twoside,12pt]{smfart}

\usepackage{smfenum}
\usepackage{smfthm}
\usepackage{amssymb}
\usepackage{mathrsfs,color,eucal}%,refcheck}
\input{xypic}

\usepackage[T1]{fontenc}
\usepackage{calligra}

\theoremstyle{plain}

\def\theequation{\fnsymbol{equation}}

\newcounter{intro}

\newtheorem{theoint}[intro]{Theorem}

\newcounter{nonumber}

\def\CC{\mathbb{C}}

\def\ZZ{\mathbb{Z}} 

\def\A{{\rm A}}
\def\B{{\rm B}}

\def\E{{\rm E}}
\def\F{{\rm F}}
\def\G{{\rm G}}
\def\H{{\rm H}}
\def\I{{\rm I}}
\def\J{{\rm J}}
\def\K{{\rm K}}
\def\L{{\rm L}}
\def\M{{\rm M}}
\def\N{{\rm N}}
\def\P{{\rm P}}
\def\Q{{\rm Q}}

\def\rS{{\rm S}}

\def\U{{\rm U}}
\def\V{{\rm V}}
\def\W{{\rm W}}

\def\Y{{\rm Y}}

\def\tG{{\widetilde{\G}}}

\def\AA{\mathfrak{A}}
\def\BB{\mathfrak{B}}
\def\HH{\mathfrak{H}}
\def\JJ{\mathfrak{J}}
\def\KK{\mathfrak{K}}

\def\MM{\mathfrak{M}}
\def\PP{\mathfrak{P}}

\def\WW{{\mathbf W}}

\def\Cc{\EuScript{C}}

\def\Gg{\mathscr{G}}
\def\Hh{\mathscr{H}}

\def\Ll{\mathscr{L}}
\def\Mm{\mathscr{M}}   %%%% !!!!

\def\Oo{\EuScript{O}}
\def\Pp{\mathscr{P}}

\def\Rr{\mathfrak{R}}

\def\b{\beta}

\def\e{\varepsilon}

\def\g{\gamma}

\def\k{\kappa}
\def\l{\lambda}

\def\p{\mathfrak{p}}

\def\s{\sigma}
\def\t{\tau}
\def\th{\theta}
\def\v{\upsilon}
\def\w{\varpi}
\def\y{\text{{\rm\Large\calligra y}\,}}

\def\ttau{\til\t}
\def\tth{\til\th}
\def\ee{\mathbf e}

\def\>{\geqslant}
\def\<{\leqslant}

\def\Hom{{\rm Hom}}
\def\End{{\rm End}}
\def\Aut{{\rm Aut}}

\def\GL{{\rm GL}}
\def\SO{{\rm SO}}
\def\O{{\rm O}}
\def\SL{{\rm SL}}
\def\Sp{{\rm Sp}}

\def\tr{{\rm tr}}

\def\Ind{{\rm Ind}}

\def\ss{\mathfrak s}

\def\La{\Lambda}
\def\Si{\Sigma}

\def\cInd{\hbox{{\rm c-Ind}}}

\def\so{{\mathsf o}}
\def\sm{{\mathsf m}}

\def\oe{{\fo_\E}}
\def\of{{\fo_\F}}
\def\pf{{\mathfrak p_\F}}
\def\fo{{\Oo}}

\def\labelenumi{{\rm(\roman{enumi})}}
\def\theenumi{\roman{enumi}}
\def\labelenumii{{\rm(\alph{enumii})}}
\def\theenumii{\alph{enumii}}

\def\ignore#1{\relax}

\def\til#1{{\widetilde{#1}}}
\def\bs#1{{\boldsymbol{#1}}}
\def\ov#1{{\overline{#1}}}
\def\({\left(}
\def\){\right)}

\def\lprime{{\ell}}

\author{Michitaka Miyauchi}
\address{Faculty of Liberal Arts and Sciences, Osaka Prefecture University,
1-1 Gakuen-cho Nakaku Sakai Osaka 599-8531 Japan} 
\email{michitaka.miyauchi@gmail.com}

\author{Shaun Stevens}
\address{School of Mathematics, University of East Anglia, Norwich Research Park,
  Norwich NR4 7TJ, United Kingdom}
\email{Shaun.Stevens@uea.ac.uk}

\title[Semisimple types for $p$-adic classical groups]{Semisimple types for $\bs p$-adic classical groups}

\begin{abstract}
We construct, for any symplectic, unitary or special orthogonal group over a 
locally compact nonarchimedean local field of odd residual characteristic, a 
type for each Bernstein component of the category of smooth representations, using
Bushnell--Kutzko's theory of covers. Moreover, for a component corresponding to 
a cuspidal representation of a maximal Levi subgroup, we prove that the Hecke 
algebra is either abelian, or a generic Hecke algebra on an infinite dihedral 
group, with parameters which are, at least in principle, computable via results 
of Lusztig.
\end{abstract}

\thanks{
This work was supported by EPSRC grants GR/T21714/01, EP/G001480/1 and EP/H00534X/1.
}

\begin{document}

\maketitle

\ \hfill{\today}

%\tableofcontents

%%%%%%%%%%%%%%%%%%%%%%%%%%%%%%%%%%%%%%%%%%%%%%%%%%%%%%%%%%%%%%%%%%%%%%%%

\section*{Introduction}

The study of the irreducible smooth (complex) representations
of~$p$-adic groups~$\G$ has seen much progress over the last fifty
years, inspired especially by the (local) Langlands programme. A basic
approach, due to Harish--Chandra, is: first classify all the
irreducible representations which do \emph{not} arise as quotients of
representations parabolically induced from representations of a proper
Levi subgroup (these are called \emph{cuspidal}); then classify all
quotients of representations parabolically induced from a
\emph{cuspidal} representation of a Levi subgroup. Because parabolic
induction does not preserve irreducibility, and because its
reducibility is related to the poles and zeros of L-functions, in
following this approach it is both necessary and interesting to study
the full (abelian) category of smooth representations~$\Rr(G)$.

A fundamental general result, for~$\G$ a connected reductive~$p$-adic
group, is the Bernstein decomposition~\cite{Be}, which
splits~$\Rr(G)$ into \emph{blocks} (indecomposable abelian
summands)~$\Rr^{\ss}(\G)$. These are indexed by (equivalence classes
of) pairs~$\ss=[\M,\tau]_\G$, with~$\M$ a Levi subgroup of~$\G$
and~$\tau$ a cuspidal irreducible representation of~$\M$, while the
irreducible objects in~$\Rr^{\ss}(\G)$ are precisely the irreducible
quotients of the parabolically induced representations
$\Ind_{\M,\P}^{\G}\tau\chi$,
for~$\P$ any parabolic subgroup with Levi factor~$\M$, and~$\chi$ any
character ($1$-dimensional representation) of~$\M$ trivial on every
compact subgroup (an \emph{unramified} character).

In~\cite{BK1}, Bushnell--Kutzko give a strategy for understanding any
block~$\Rr^{\ss}(\G)$: one seeks to construct a pair~$(\J,\l)$ (called
an~\emph{$\ss$-type}), consisting of a compact open subgroup~$\J$
of~$\G$ and an irreducible (smooth) representation~$\l$ of~$\J$, which
characterizes the block in the sense that the irreducible objects
in~$\Rr^{\ss}(\G)$ are exactly the irreducible representations~$\pi$
of~$\G$ such that~$\Hom_{\J}(\pi,\l)\ne 0$. (We say that~$\pi$
\emph{contains~$\l$}.) Then the block~$\Rr^{\ss}(\G)$ is equivalent to
the category of modules over the spherical Hecke
algebra~$\Hh(\G,\l)=\End_{\G}(\cInd_\J^\G\l)$, so we are reduced to
computing~$\Hh(\G,\l)$ and its modules. Moreover, Bushnell--Kutzko's theory of~\emph{covers}, which we recall below, gives a technique for trying to construct types (and their Hecke algebras) for general~$\M$ from those in the cuspidal case (that is, when~$\M=\G$). 

This programme has been carried out in its entirety for the
groups~$\GL_N$ and its inner forms, $\SL_N$, and, when the residual
characteristic~$p$ is odd,~$\U(2,1)$ and~$\Sp_4$. It has also been
completed for an arbitrary connected reductive group for~\emph{level
  zero} blocks, that is, for~$[\M,\tau]_\G$ where~$\tau$ contains the
trivial representation of the pro-$p$-radical of some parahoric
subgroup. For inner forms of~$\GL_N$, the Hecke algebras which arise
are all tensor products of generic Hecke algebras of type~$\A$;
for~$\SL_N$ one gets a similar algebra tensored with the group algebra
of a finite group, but twisted by a cocycle.

In this paper, we largely complete the programme for an arbitrary
classical group~$\G$ when the residual characteristic is odd. More
precisely, let~$\F_\so$ be a locally compact nonarchimedean local
field with residue field of odd cardinality~$q_\so$, and let~$\G$ be
the group of rational points of a symplectic, special orthogonal or
unitary group defined over~$\F_\so$. Our first main result is:

\begin{theoint}\label{thm:main}
Let $\M$ be a Levi subgroup of $\G$, let $\tau$ be a cuspidal
irreducible representation of $\M$, and put $\ss=[\M,\tau]_\G$. There
is an $\ss$-type $(\J,\l)$ which is, moreover, a cover of the
$\ss_\M$-type $(\J\cap\M,\l|\J\cap\M)$.
\end{theoint}

At present, we are only able to determine the Hecke algebra in the
case of a \emph{maximal} proper Levi subgroup (though see the comments
below for some implication in other cases); it turns out that the Hecke
algebras which arise are as for the group ${\rm Sp}_4(\F)$~\cite{BB2},
although there are more possibilities for the parameters. 
We also remark that this case
of a maximal Levi subgroup is the most interesting in terms of
implications on poles and zeros of L-functions; in particular, it is
possible to use the results here to compute explicitly the cuspidal
representations in an L-packet. 

\begin{theoint}\label{thm:maximal}
In the situation of Theorem~\ref{thm:main}, suppose moreover that~$\M$
is a maximal proper Levi subgroup of~$\G$ and write~$\N_\G(\ss_\M)$
for the set of~$g\in\G$ such that~$g$ normalizes~$\M$ and~${}^g\tau$
is equivalent to~$\tau\chi$, for some unramified character~$\chi$ of~$\M$.
\begin{enumerate}
\item If~$\N_\G(\ss_\M)=\M$ then the Hecke algebra~$\Hh(\G,\l)$ is
abelian, isomorphic to~$\CC\left[X^{\pm 1}\right]$.
\item If~$\N_\G(\ss_\M)\ne\M$ then the Hecke algebra~$\Hh(\G,\l)$ is a
generic Hecke algebra on an infinite dihedral group; that is, it is
generated by~$T_0,T_1$, each invertible and supported on a single
double coset, with relations
$$
(T_i-q_i)(T_i+1)=0,
$$
for some integer~$q_i\in q_\so^\ZZ$. 
\end{enumerate}
\end{theoint}

Moreover, in~\S\ref{S.hecke}, we give a recipe which reduces the
calculation of the parameters~$q_i$ in this Hecke algebra to the
computation of a certain quadratic character (which is sometimes known
to be trivial) and of the parameters in two finite Hecke algebras,
which are computable through the work of Lusztig~\cite{L}. We explore
certain cases of this further in work in progress%~\cite{BHS,LS}
, though we emphasise that
the computation of the quadratic character appears, in general, to be
a very subtle matter: see the work of Blondel~\cite{Bl} for more on this.

We also remark that, for symplectic groups, the propagation results of Blondel~\cite{Bl0} together with our Theorem~\ref{thm:maximal} now give the Hecke algebra when~$\M\simeq\GL_r(\F)^s\times\Sp_{2N}(\F)$ and~$\tau=\til\tau^{\otimes s}\otimes\tau_0$. Whether the results there and here could be pushed to give a description of the Hecke algebra in the general case is not clear.

\medskip

We now describe the proofs so we suppose we are in the situation of
Theorem~\ref{thm:main}. The class~$\ss=[\M,\tau]_\G$ determines a (cuspidal) class~$\ss_\M=[\M,\tau]_\M$ for~$\M$, which gives us a block~$\Rr^{\ss_\M}(\M)$ of the category of smooth representations of~$\M$. An~$\ss_\M$-type~$(\J_\M,\l_\M)$ was constructed by the second author in~\cite{S5} (though we take the opportunity here to correct some inaccuracies -- see~\S\ref{S.cuspidal}). We say that a pair~$(\J,\l)$ is \emph{decomposed} over~$(\J_\M,\l_\M)$ if, for any parabolic subgroup~$\P=\M\U$ with Levi factor~$\M$,
\begin{enumerate}
\item $\J$ has an Iwahori decomposition with respect to~$(\M,\P)$ and~$\J\cap\M=\J_\M$; and
\item $\l$ restricts to~$\l_\M$ on~$\J_\M$, and to a multiple of the trivial representation on~$\J\cap\U$.
\end{enumerate}
If a further technical condition on the Hecke algebra~$\Hh(\J,\l)$ is satisfied (it contains an invertible element supported only on the double coset of a strongly positive element of the centre of~$\M$) then~$(\J,\l)$ is a \emph{cover} of~$(\J_\M,\l_\M)$, in which case it is also an~$\ss$-type. Moreover, one gets an embedding of Hecke algebras~$\Hh(\M,\l_\M)\hookrightarrow\Hh(\G,\l)$ and, in certain circumstances, one can also deduce the rank (and other structure) of~$\Hh(\G,\l)$ as an~$\Hh(\M,\l_\M)$-module.

To construct a cover, we do \emph{not} in fact start with the type~$(\J_\M,\l_\M)$ but rather construct~$(\J,\l)$ directly, then observing that it is a cover of its restriction to~$\M$, which is indeed an~$\ss_\M$-type. To this end, the starting point is a result of
Dat~\cite{Dat}, building on work of the second author in~\cite{S4}. In
the latter paper, so-called \emph{semisimple characters} of certain
compact open subgroups of~$\G$ were constructed, generalizing
constructions of Bushnell--Kutzko~\cite{BK}. These come in families
indexed by a semisimple element~$\b$ of the Lie algebra of~$\G$ and a
lattice sequence~$\La$, which can be interpreted as a point in the
building of the centralizer~$\G_\b$ of~$\b$ via~\cite{BrS}. 

Dat proved that, given~$\ss=[\M,\tau]_\G$ as above, there is
a \emph{self-dual semisimple character}~$\th$ of a compact open
subgroup~$\H^1$ of~$\G$ such that~$(\H^1,\th)$ is a decomposed pair
over~$(\H^1\cap\M,\th|_{\H^1\cap\M})$ and~$\tau$
contains~$\th|_{\H^1\cap\M}$. There is considerable flexibility here;
in particular, the associated lattice sequence~$\La$ may be chosen so
that the parahoric subgroup it defines in~$\G_\b$ (that is, the
stabilizer of the point it defines in the building) also has an
Iwahori decomposition with respect to any parabolic subgroup with Levi
factor~$\M$. (Indeed, this is generically the case.) The element~$\b$ also 
has a Levi subgroup~$\L$ attached to it (the minimal Levi subgroup containing~$\G_\b$) and we have~$\M\subseteq\L$.

Our first task is to extend the constructions of~\cite{S4,S5} to the self-dual case, in particular the notion of \emph{(standard)~$\b$-extension~$\k$} and its realization as an induced representation~$\Ind_{\J_\P}^\J\k_\P$, for~$\P$ a parabolic subgroup with Levi component~$\M$. The main property here is that the representation~$\k_\M:=\k_\P|_{\J_\P\cap\M}$ of~$\J_\M:=\J_\P\cap\M$ is a (standard)~$\b$-extension in~$\M$, in the sense of~\cite{S5}, with extra compatibility properties coming from conjugation in~$\L$; indeed, it is ensuring these compatibilities which would make it difficult to start with a type in~$\M$ and build a cover from it. %Moreover, the pair~$(\J_\P,\k_\P)$ is decomposed over~$(\J_\M,\k_\M)$. 

Now our cuspidal representation~$\tau$ of~$\M$ contains a representation of~$\J_\M$ of the form~$\l_\M=\k_\M\otimes\rho_\M$, for~$\rho_\M$ the inflation of a cuspidal representation of the (possibly disconnected) finite reductive quotient~$\J_\M/\J_\M^1$, and~$(\J_\M,\l_\M)$ is an~$\ss$-type. Since~$\J_\P/\J_\P^1\simeq\J_\M/\J_\M^1$, we can also form the representation~$\l_\P=\k_\P\otimes\rho_\M$ and the claim is then that~$(\J_\P,\l_\P)$ is a cover of~$(\J_\M,\l_\M)$. There is a small but important subtlety here: it is in fact the inverse image~$\J_\M^\so$ of the connected component of~$\J_\M/\J_\M^1$ that we work with, along with a representation~$\l_\M^\so=\k_\M\otimes\rho_\M^\so$ contained in~$\l_\M$, and we prove that~$(\J^\so_\P,\l^\so_\P)$ is a cover of~$(\J^\so_\M,\l^\so_\M)$. That~$(\J_\P,\l_\P)$ is also a cover follows from a result of Morris: this phenomenon already arises for level zero representations.

The proof uses transitivity of covers, showing that~$(\J_\P^\so,\l_\P^\so)$ is a cover of~$(\J_\P^\so\cap\M',\l_\P^\so|_{\J_\P^\so\cap\M'})$ for a chain of Levi subgroups~$\M'$ ending with~$\M$. The first step is with~$\M'=\L$, which is straightforward by consideration of intertwining; indeed, the embedding of Hecke algebras in this case is an isomorphism. This reduces us to the case~$\L=\G$, which is the case of a~\emph{skew} semisimple character considered in~\cite{S5}, and the rest of the argument is essentially contained there. By intertwining arguments, we reduce to the case in which there is no proper Levi subgroup of~$\G$ containing the normalizer of~$\rho_\M^\so|_{\J_\M^\so}$. Finally, we pull off the remaining blocks of~$\M$ one at a time; that is, we go in steps with~$\M'=\GL_r(\F)\times\G^0$ a maximal proper Levi subgroup of~$\G$ containing~$\M=\GL_r(\F)\times\M^0$, with~$\M^0$ a Levi subgroup of the classical group~$\G^0$. (In the case of even special orthogonal groups we must sometimes remove blocks in pairs.)

The final step is achieved by producing Hecke algebra embeddings~$\Hh(\Gg_i,\rho_\M^\so\chi_i)\hookrightarrow\Hh(\G,\l_\P^\so)$, for~$i=0,1$, where~$\Gg_i$ is a finite reductive group having~$\J_\M^\so/\J_\M^1$ as a maximal proper Levi subgroup, and~$\chi_i$ is a quadratic character. Each of these finite Hecke algebras is two-dimensional, generated by an element~$T_i$ which is supported on a single double-coset and satisfies a quadratic relation. It is a power of the product of the images of~$T_i$ in~$\Hh(\G,\l_\P^\so)$ which gives the required invertible element of the Hecke algebra. 

In the case that~$\M$ is maximal and~$\N_\G(\ss_\M)\ne\M$, the same argument allows one to describe the Hecke algebra of the cover completely: the images of the two embeddings together generate~$\Hh(\G,\l_\P^\so)$ and there are no further relations by support considerations. Again, there are some additional complications arising from the fact that the finite groups~$\Gg_i$ (which are the reductive quotients of non-connected parahoric subgroups in~$\G_\b$) need not be connected; some care is needed in dealing with these.

\medskip

Finally we summarize the contents of the various sections. The basic objects involved in the construction are recalled in section~\ref{S.notation}, while section~\ref{S.characters} extends the various constructions from the skew case in~\cite{S5} to the case of a self-dual semisimple character. In section~\ref{S.cuspidal} we recall the construction of types in the cuspidal case (and make some corrections), before constructing the cover and proving Theorem~\ref{thm:main} in section~\ref{S.covers}. Finally, the computation of the Hecke algebra is given in section~\ref{S.hecke}.

\subsection*{Acknowledgements} 
This paper has been a long time coming. The second author would like to thank Muthu Krishnamurthy for asking the question which prompted him to get on and finish it. He would also like to thank Laure Blasco and Corinne Blondel for point out some mistakes in~\cite{S5} and especially Corinne Blondel for many useful discussions.

%%%%%%%%%%%%%%%%%%%%%%%%%%%%%%%%%%%%%%%%%%%%%%%%%%%%%%%%%%%%%%%%%%%%%%%%

\section{Notation and preliminaries}\label{S.notation}

Let $\F$ be a nonarchimedean locally compact field of odd residual
characteristic. Let $\l\mapsto \overline{\l}$ denote a (possibly
trivial) galois involution on $\F$ with fixed field $\F_\so$. For $\K$ a
finite extension of $\F_\so$, we denote by $\Oo_{\K}$ its ring of
integers, by $\p_{\K}$ the maximal ideal of $\Oo_{\K}$ and by $k_\K$
its residue field. We also denote by $e(\K/\F_\so)$ and $f(\K/\F_\so)$ the
ramification index and residue class degree of $\K/\F_\so$
respectively, and put $\e_\F=(-1)^{e(\F/\F_\so)+1}$.

We fix $\w_\F$ a uniformizer of $\F$ such that
$\overline{\w_\F}=\e_\F\w_\F$, and put
$\w_\so=\w_\F^{e(\F/\F_\so)}$, a uniformizer of $\F_\so$. We also fix
$\psi_\so$, a character of the additive group of $\F_\so$ with
conductor $\mathfrak p_{\F_\so}$; then we put 
$\psi_\F=\psi_\so\circ\tr_{\F/\F_\so}$, a character of the additive
group of $\F$ with conductor $\pf$. We also denote by~$f\mapsto\ov f$
the involution induced on the polynomial ring~$\F[X]$.

For $u$ a real number, we denote by $\lceil{u}\rceil$ the smallest
integer which is greater than or equal to $u$, and by
$\lfloor{u}\rfloor$ the greatest integer which is smaller than or
equal to $u$, that is, its integer part.

All representations considered here are smooth and complex.

The material of this section is essentially a summary of necessary
definitions and basic results. More details can be found
in~\cite{BK,S4}.

%%%%%%%%%%%%%%%%%%%%%%%%%%%%%%%%%%%%%%%%%%%%%%%%%%%%%%%%%%%%%%%%%%%%%%%%
\subsection{}
Let $\e=\pm 1$ and let $\V$ be a finite-dimensional $\F$-vector space
equipped with a nondegenerate $\e$-hermitian form $h$: thus
$$
\l h(v,w)\ =\ h(\l v,w)\ =\ \e\overline{h(w,\l v)},\qquad v,w\in\V,\ \l\in\F.
$$
Put $\A=\End_\F(\V)$, an $\F$-split simple central $\F$-algebra
equipped with the adjoint anti-involution $a\mapsto \overline a$
defined by
$$
h(av,w) = h(v,\overline a w),\qquad v,w\in\V;
$$
this anti-involution coincides with the galois involution on the
naturally embedded copy of $\F$ in $\A$.

%%%%%%%%%%%%%%%%%%%%%%%%%%%%%%%%%%%%%%%%%%%%%%%%%%%%%%%%%%%%%%%%%%%%%%%%
\subsection{}
Set~$\tG=\Aut_\F(\V)$ and let~$\s$ be the involution given
by~$g\mapsto \overline g^{-1}$, for~$g\in\tG$. We also have an action
of~$\s$ on the Lie algebra~$\A$ given by~$a\mapsto-\overline a$,
for~$a\in\A$. We put~$\Si=\{1,\s\}$, where~$1$ acts as the identity on
both~$\tG$ and~$\A$.

Put $\G^+=\tG^\Sigma=\{g\in\tG:h(gv,gw)=h(v,w)$ for all $v,w\in V\}$,
the $\F_\so$-points of a unitary, symplectic or orthogonal group $\bs\G^+$ 
over $\F_\so$. Let $\G$ be the $\F_\so$-points of the connected component 
$\bs\G$ of $\bs\G^+$, so that $\G=\G^+$ except in the orthogonal case. 
Put $\A_-=\A^\Sigma$, the Lie algebra of $\G$. In general, for 
$\rS$ a subset of $\A$, we will write $\rS_-$ or $\rS^-$ for
$\rS\cap\A_-$, and, for $\til\H$ a subgroup of $\tG$, we will write
$\H$ for $\til\H\cap\G$. 

If~$\F=\F_\so$,~$\e=+1$,~$\dim_\F\V=2$ and~$h$ is isotropic,
then~$\G\simeq\SO(1,1)(\F)\simeq\GL_1(\F)$ so is
well-understood. Consequently, we exclude this case. In particular,
the centre of~$\G^+$ is the naturally embedded copy
of~$\F^1:=\{\l\in\F:\l\ov\l=1\}$, which is compact.

%%%%%%%%%%%%%%%%%%%%%%%%%%%%%%%%%%%%%%%%%%%%%%%%%%%%%%%%%%%%%%%%%%%%%%%%
\subsection{}
An \emph{$\Oo_{\F}$-lattice sequence} on $\V$ is a map 
$$
\La:\ZZ \to \{\hbox{$\Oo_\F$-lattices in $\V$}\}
$$ 
which is decreasing (that is, $\La(k)\supseteq\La(k+1)$ for all
$k\in\ZZ$) and such that there exists a positive integer
$e=e(\La|\Oo_\F)$ satisfying $\La(k+e)=\p_\F\La(k)$, for all $k\in\ZZ$. 
This integer is called the \emph{$\Oo_\F$-period} of $\La$.
If $\La(k)\supsetneq\La(k+1)$ for all $k\in\ZZ$, then the lattice
sequence $\La$ is said to be \emph{strict}. If
$\dim_{k_\F}\La(k)/\La(k+1)$ is independent of $k$, we say that the
lattice sequence is \emph{regular}.

Associated with an $\Oo_{\F}$-lattice sequence $\La$ on $\V$, we have an
$\Oo_{\F}$-lattice sequence on $\A$ defined by
$$
k\mapsto\PP_{k}(\La)=\{a\in\A : a\La(i)\subseteq\La(i+k),\ i\in\ZZ\},
\quad k\in\ZZ.
$$
The lattice $\AA(\La)=\PP_0(\La)$ is a hereditary $\Oo_\F$-order in $\A$, 
and $\mathfrak{P}(\La)=\PP_1(\La)$ is its Jacobson radical;
these two lattices depend only on the set $\{\La(k): k\in\ZZ\}$. 

\medskip

We denote by $\KK(\La)$ the \emph{$\til\G$-normalizer} of $\La$: 
that is, the subgroup of $\til\G$ made of all elements 
$g$ for which there is an integer $n\in\ZZ$ such that 
$g(\La(k))=\La(k+n)$ for all $k\in\ZZ$.
Given $g\in\KK(\La)$, such an integer 
is unique: it is denoted 
$\v_{\La}(g)$ and called the \emph{$\La$-valuation} of $g$. 
This defines a group homomorphism $\v_{\La}$ from $\KK(\La)$ to $\ZZ$.
Its kernel, denoted $\til\P(\La)$, is the group of invertible elements 
of $\AA(\La)$. 
We set $\til\P_0(\La)=\til\P(\La)$ and, for $k\>1$, we set 
$\til\P_k(\La)=1+\PP_k(\La)$.

%%%%%%%%%%%%%%%%%%%%%%%%%%%%%%%%%%%%%%%%%%%%%%%%%%%%%%%%%%%%%%%%%%%%%%%%
\subsection{}
Given $\La$ an $\of$-lattice sequence, the \emph{affine class} of
$\La$ is the set of all $\Oo_{\F}$-lattice sequences on $\V$ of the form:
$$
a\La+b:k\mapsto\La(\lceil(k-b)/a\rceil),
$$
with $a,b\in\ZZ$ and $a\>1$. The $\Oo_\F$-period of $a\La+b$ is $a$
times the period $e(\La|\Oo_\F)$ of $\La$. Note that
$$
\PP_k(a\La+b)=\PP_{\lceil k/a\rceil}(\La)
$$
so that changing $\La$ in its affine class only changes $\PP_k(\La)$
in its affine class, indeed only by a scale in the indices; similarly,
$\til\P_k(\La)$ is only changed by a scale in the indices, while
$\KK(a\La+b)=\KK(\La)$. 

%%%%%%%%%%%%%%%%%%%%%%%%%%%%%%%%%%%%%%%%%%%%%%%%%%%%%%%%%%%%%%%%%%%%%%%%
\subsection{}
We call an $\of$-lattice sequence $\La$ \emph{self-dual\/} if there 
exists $d\in\ZZ$, such that $\{v\in\V:h(v,\La(k))\subseteq\pf\}=\La(d-k)$ 
for all $k\in\ZZ$. By changing a self-dual $\of$-lattice 
sequence in its affine class, we may and do normalize all self-dual
lattice sequences so that $d=1$ and $e(\La|\of)$ is even. 

For $\La$ a self-dual lattice sequence, the $\of$-lattices
$\PP_k(\La)$ are stable under the involution $\s$ (on
$\A$). Similarly, the groups $\til\P_k$ are fixed by $\s$ (on
$\til\G$) and we put $\P^+=\P^+(\La)=\til\P\cap \G^+$, a compact
open subgroup of $\G^+$, and $\P=\P(\La)=\P^+\cap \G$. We have a
filtration of $\P(\La)$ by normal subgroups
$\P_k=\P_k(\La)=\til\P_k^\Sigma=\til\P_k\cap \G$, for $k>0$.
 We also have, for $k>0$, a bijection $\PP_k^-(\La) \to 
\P_k$ given by the Cayley map $x\mapsto (1+\frac x2)(1-\frac
x2)^{-1}$, which is equivariant under conjugation by $\P$.

The quotient group $\Gg=\P/\P_1$ is (the group of rational points
of) a reductive group over the finite field $k_{F_\so}$. However, it is
not, in general, connected. We denote by $\P^\so=\P^\so(\La)$ the
inverse image in $\P$ of (the group of rational points
of) the connected component $\Gg^\so$ of $\Gg$; then $\P^\so$ is a
parahoric subgroup of $\G$.

%%%%%%%%%%%%%%%%%%%%%%%%%%%%%%%%%%%%%%%%%%%%%%%%%%%%%%%%%%%%%%%%%%%%%%%%
\subsection{}
A \emph{stratum} in $\A$ is a quadruple $[\La,n,m,\b]$ made of an 
$\Oo_\F$-lattice sequence $\La$ on $\V$, 
two integers $m,n$ such that $0\<m\<n$, and an element 
$\b\in\PP_{-n}(\La)$.
Two strata $[\La,n,m,\b_i]$, for $i=1,2$, in $\A$ are said to be
\emph{equivalent} if $\b_2-\b_1\in\PP_{-m}(\La)$. A stratum
$[\La,n,m,\b]$ is called \emph{null} if it is equivalent to
$[\La,n,m,0]$, that is, if $\b\in\PP_{-m}(\La)$.

A stratum
$[\La,n,m,\b]$ is called \emph{self-dual\/} if $\Lambda$ is self-dual and
$\b\in \A_-$. (Note that this notion has been called \emph{skew} in
previous papers; here we reserve the term skew for a more precise
situation -- see~\S\ref{S.characters}.)

For $n\ge m\ge \frac n2>0$, an equivalence class of 
strata corresponds to a character of $\til\P_{m+1}(\La)$, by
$$
[\La,n,m,\b]\mapsto (\til\psi_\b:x\mapsto\psi_\F\circ\tr_{\A/\F}(\b(x-1)),
\hbox{ for }x\in \til\P_{m+1}(\La)),
$$
while an equivalence class of self-dual strata corresponds to a
character of $\P_{m+1}(\La)$, by
$$
[\La,n,m,\b]\mapsto \psi_\b=\til\psi_\b|_{\P_{m+1}(\La)}.
$$
A null stratum corresponds to the trivial character.

%%%%%%%%%%%%%%%%%%%%%%%%%%%%%%%%%%%%%%%%%%%%%%%%%%%%%%%%%%%%%%%%%%%%%%%%
\subsection{} 
For $[\La,n,m,\b]$ a stratum in $\A$, we set
$$
\y=\y(\b,\La)=\varpi_\F^{n/g}\b^{e/g},
$$
where $e=e(\La|\Oo_\F)$ and $g=\gcd(n,e)$. The
characteristic polynomial of $\y+\PP_1(\La)$ (considered as an element
of $\AA(\La)/\PP_1(\La)$) is called the \emph{characteristic
polynomial $\varphi_\b(X)\in k_\F[X]$ of the stratum
$[\La,n,m,\b]$}. The stratum $[\La,n,m,\b]$ is said to be \emph{split}
if $\varphi_\b(X)$ has (at least) two distinct irreducible factors.

If $[\La,n,m,\b]$ is self-dual then we have $\ov\y=\e_\b\y$, where
$\e_\b=\e_\F^{n/g}(-1)^{e/g}$, and thus
$\varphi_\b(X)=\ov\varphi_\b(\e_\b X)$. We say that the stratum is
\emph{$\G$-split} if $\varphi_\b(X)$ has an irreducible factor
$\psi(X)$ such that $\psi(X),\ov\psi(\e_\b X)$ are coprime.

%%%%%%%%%%%%%%%%%%%%%%%%%%%%%%%%%%%%%%%%%%%%%%%%%%%%%%%%%%%%%%%%%%%%%%%%
\subsection{}
Let $\E$ be a finite extension of $\F$ contained in $\A$. An
$\Oo_\F$-lattice sequence $\La$ on $\V$ is said to be 
\emph{$\E$-pure} if it is normalized by $\E^{\times}$, in which case
it is also an $\Oo_\E$-lattice sequence. Denote by $\B=\End_\E(\V)$
the centralizer of $\E$ in $\A$ and by $\La_{\Oo_\E}$ the lattice
sequence $\La$ considered as an $\Oo_\E$-lattice sequence.

%%%%%%%%%%%%%%%%%%%%%%%%%%%%%%%%%%%%%%%%%%%%%%%%%%%%%%%%%%%%%%%%%%%%%%%%
\subsection{}
Given a stratum $[\La,n,m,\b]$ in $\A$, we denote by $\E$ the 
$\F$-algebra generated by $\b$. 
This stratum is said to be \emph{pure} if $\E$ is a field, 
if $\La$ is $\E$-pure and if $\v_{\La}(\b)=-n$.
Given a pure stratum $[\La,n,m,\b]$, we denote by $\B$ the centralizer 
of $\E$ in $\A$.
For $k\in\ZZ$, we set:
\begin{equation*}
\mathfrak{n}_k(\b,\La)=\{x\in\AA(\La)\ |\ \b x-x\b\in\PP_k(\La)\}.
\end{equation*}
The smallest integer $k\>\v_{\La}(\b)$ such that $\mathfrak{n}_{k+1}(\b,\La)$ 
is contained in $\AA(\La)\cap\B+\PP(\La)$ is called the 
\emph{critical exponent} of the stratum $[\La,n,m,\b]$, denoted 
$k_0(\b,\La)$.

The stratum $[\La,n,m,\b]$ is said to be \emph{simple} if it is pure
and if we have $m<-k_0(\b,\La)$. 

Given $n\> 0$ and $\La$ an $\Oo_\F$-lattice sequence, there is another
stratum which plays a very similar role to simple strata, namely the
\emph{zero stratum} $[\La,n,n,0]$. (Note that this was called a null
stratum in~\cite{S4,S5}.)

%%%%%%%%%%%%%%%%%%%%%%%%%%%%%%%%%%%%%%%%%%%%%%%%%%%%%%%%%%%%%%%%%%%%%%%%
\subsection{} 
Let $[\La,n,m,\b]$ be a stratum in $\A$ and suppose we have a
decomposition $\V=\bigoplus_{i\in\I} \V^i$ into $\F$-subspaces. Let
$\La^i$ be the lattice sequence on $\V^i$ given by $\La^i(k)=\La(k)\cap\V^i$
and put $\b_i= \ee^i\b\ee^i$, where $\ee^i$ is the projection
onto $\V^i$ with kernel $\bigoplus_{j\ne i}\V^j$. We use the block notation
$\A^{ij}=\Hom_\F(\V^j,\V^i)$.

We say that $\V=\bigoplus_{i\in\I}\V^i$ is a \emph{splitting\/} for
$[\La,n,m,\b]$ if $\La(k)=\bigoplus_{i\in\I} \La^i(k)$, for all
$k\in\ZZ$, and $\b=\sum_{i\in\I}\b_i$.

Suppose $\V=\bigoplus_{i\in\I}\V^i$ and $\V=\bigoplus_{j\in\J}\W^j$ are two
decompositions of $\V$. We say that $\bigoplus_{i\in\I}\V^i$ is a
\emph{refinement} of $\bigoplus_{j\in\J}\W^j$ (or
$\bigoplus_{j\in\J}\W^j$ is a \emph{coarsening} of
$\bigoplus_{i\in\I}\V^i$) if, for each $i\in\I$, there exists $j\in\J$
such that $\V^i\subseteq\W^j$.

%%%%%%%%%%%%%%%%%%%%%%%%%%%%%%%%%%%%%%%%%%%%%%%%%%%%%%%%%%%%%%%%%%%%%%%%
\subsection{}
A stratum $[\La,n,m,\b]$ in $\A$ is called \emph{semisimple\/} if
either it is a zero stratum or $\b\not\in\PP_{1-n}(\La)$ and there is a splitting 
$\V=\bigoplus_{i\in\I}\V^i$ for the stratum such that 
\begin{enumerate}
\item for $i\in\I$, $[\La^i,q_i,m,\b_i]$ is a
simple or zero stratum in $A^{ii}$, where $q_i=m$ if $\b_i=0$,
$q_i=-\v_{\La^i}(\b_i)$ otherwise; and 
\item for $i,j\in\I$, $i\ne j$, the stratum 
$[\La^i\oplus\La^j,q,m, \b_i+\b_j]$ is not equivalent
to a simple or zero stratum, with $q=\max \{q_i,q_j\}$.
\end{enumerate}
In this case, the splitting is uniquely determined (up to ordering) by
the stratum and we put $\Ll_\b=\bigoplus_{i\in\I} \A^{ii}$.
We put $\E=\F[\b]=\bigoplus_{i\in\I}\E_i$,
where $\E_i=\F[\b_i]$. We will sometimes write ``$\La$ is an
$\oe$-lattice sequence'' to mean that $\La=\bigoplus_{i\in\I}\La^i$ and
each $\La^i$ is an $\fo_{\E_i}$-lattice sequence on $\V^i$.

Let $\B=\B_\b$ denote the $\A$-centralizer of $\b$, so that
$\B=\bigoplus_{i\in\I}\B_i$, where $\B_i$
is the centralizer of $\b_i$ in $\A^{ii}$. 
We write $\til\G_\E=\B^\times$, $\til\G^i=\Aut_\F(\V^i)$ and
$\til\G_{\E_i}=\B_i^\times=\til\G^i\cap\til\G_\E$, so that
$\til\L_\b=\Ll_\b^\times=\prod_{i\in\I}\til\G^i$ is a Levi subgroup of $\til\G$ and
$\til\G_\E=\prod_{i\in\I}\til\G_{\E_i}\subseteq\til\L_\b$.  
Each $\til\G_{\E_i}$ is (the group of $\F_\so$-points of) the restriction of
scalars to $\F_\so$ of a general linear group over $\E_i$, provided
$\E_i/\F$ is separable; in any case,~$\til\G_{\E_i}$ is isomorphic to
some~$\GL_{m_i}(\E_i)$.
We also write $\PP_k(\La_\oe)=\PP_k(\La)\cap \B$, for $k\in\ZZ$, which gives
the filtration induced on $\B$ by thinking of $\La$ as an $\oe$-lattice
sequence, and $\til\P_k(\La_\oe)=\til\P_k(\La)\cap\B$, for $k\ge 0$.

%%%%%%%%%%%%%%%%%%%%%%%%%%%%%%%%%%%%%%%%%%%%%%%%%%%%%%%%%%%%%%%%%%%%%%%%
\subsection{} 
Let $[\La,n,m,\b]$ be a semisimple stratum in $\A$.
The \emph{affine class} of the stratum $[\La,n,m,\b]$ is the set of
all (semisimple) strata of the form
$$
[\La',n',m',\b],
$$
where $\La'=a\La+b$ is in the affine class of $\La$, $n'=an$ and $m'$
is any integer such that $\lfloor m'/a\rfloor=m$. 
In the course of the paper, there will be several objects associated
to a semisimple stratum $[\La,n,m,\b]$, in particular semisimple characters
(see~\S\ref{S.characters}). By a straightforward induction
(cf.~\cite[Lemma~2.2]{BSS}), these objects depend only on the affine
class of the stratum.

%%%%%%%%%%%%%%%%%%%%%%%%%%%%%%%%%%%%%%%%%%%%%%%%%%%%%%%%%%%%%%%%%%%%%%%%
%%%%%%%%%%%%%%%%%%%%%%%%%%%%%%%%%%%%%%%%%%%%%%%%%%%%%%%%%%%%%%%%%%%%%%%%
%%%%%%%%%%%%%%%%%%%%%%%%%%%%%%%%%%%%%%%%%%%%%%%%%%%%%%%%%%%%%%%%%%%%%%%%
%%%%%%%%%%%%%%%%%%%%%%%%%%%%%%%%%%%%%%%%%%%%%%%%%%%%%%%%%%%%%%%%%%%%%%%%

\section{Self-dual semisimple characters}\label{S.characters}

In this section we recall the notion of \emph{self-dual} semisimple
strata and characters from~\cite{Dat}, generalizing the \emph{skew}
semisimple case from~\cite{S4}. We also develop the theory of
$\b$-extensions in the self-dual situation. The results here are the
expected generalizations of the results in the skew case
from~\cite{S4,S5}. Moreover, most of the proofs follow by taking
fixed points under the involution $\s$ so are essentially identical to
those in the skew case; we will only give details when new phenomena arise.

%%%%%%%%%%%%%%%%%%%%%%%%%%%%%%%%%%%%%%%%%%%%%%%%%%%%%%%%%%%%%%%%%%%%%%%%
\subsection*{}{\bf{Self-dual semisimple strata}}
%%%%%%%%%%%%%%%%%%%%%%%%%%%%%%%%%%%%%%%%%%%%%%%%%%%%%%%%%%%%%%%%%%%%%%%%
\subsection{}
Let $[\La,n,m,\b]$ be a semisimple stratum and denote by $\V =
\bigoplus_{i\in\I}\V^i$ the  
associated splitting and use all the notations introduced
in~\S\ref{S.notation}. If $\Psi_i(X)\in\F[X]$ denotes the minimum 
polynomial of $\b_i$ then, by \cite[Remark~3.2(iii)]{S4},
we have $\V^i=\ker\Psi_i(\b)$.

If $[\La,n,m,\b]$ is also self-dual then, for each $i\in\I$,
there is a unique $j=\s(i)\in\I$ such that $\ov{\b_i} =
-\b_{j}$. Moreover, we see that $\ov\Psi_i(X) = \Psi_{\s(i)}(-X)$,
whence $(\V^i)^\perp = \bigoplus_{j \neq \s(i)}V^j$. Then, using the
usual block notation in $\A$, the action of the involution $\ov{\phantom{a}}$ on
$\A$ is such that $\ov{\A^{ij}}=\A^{\s(j)\s(i)}$.

We set $\I_0=\{i\in\I\mid \s(i)=i\}$ and choose a set of
representatives $\I_+$ for the orbits of $\s$ in
$\I\setminus\I_0$. Then we will write $\I_-=\s(\I_+)$ so that
$\I=\I_-\cup\I_0\cup\I_+$ (disjoint union) and
$$
\V = \bigoplus_{i\in\I_+}(\V^i \oplus \V^{\s(i)}) \oplus \bigoplus_{i\in\I_0} \V^{i}.
$$
It will sometimes be useful to place on ordering on $\I_+$, in which
case we will write $\I_+=\{1,\ldots,l\}$ and put $\s(i)=-i\in\I_-$, for
$i\in\I_+$; in this case we will write $\V^0=\bigoplus_{i\in\I_0}
\V^{i}$ so that $\V=\bigoplus_{i=-l}^l\V^i$, which we call the 
\emph{self-dual decomposition} associated to~$[\La,n,m,\b]$. We will 
also put $\b_0=\ee^0\b\ee^0$, where $\ee^0$ is the projection onto 
$\V^0$ with kernel $\bigoplus_{j\ne 0}\V^j$.

%%%%%%%%%%%%%%%%%%%%%%%%%%%%%%%%%%%%%%%%%%%%%%%%%%%%%%%%%%%%%%%%%%%%%%%%
\subsection{}\label{S.semigroups}
Let $[\La,n,m,\b]$ be a self-dual semisimple stratum and
$\V=\bigoplus_{i=-l}^l\V^i$ as above, with $\V^0=\bigoplus_{i\in\I_0}
\V^{i}$. We put $\til\G^i=\Aut_\F(\V^i)$, $\L^+_\b=\left(\prod_{i=-l}^l
\til\G^i\right)\cap\G^+$ and $\L_\b=\L^+_\b\cap\G$, which is a Levi subgroup of
$\G$. We have $\L_\b=\G^0\times\prod_{i=1}^l\til\G^i$, where $\G^0$ is
the unitary, symplectic or special orthogonal group fixing the
nondegenerate form $h|_{\V^0\times\V^0}$.

Put $\til\G_E=\B^\times$, the centralizer of $\b$, as
in~\S\ref{S.notation}. We put $\G^+_\E=\til\G_E\cap\G^+$ and
$\G_\E=\til\G_E\cap\G$, so that $\G_\E\subseteq\L_\b$. For $i\in\I_0$,
the involution on~$\F$ extends to each~$\E_i$ and we
write~$\E_{i,\so}$ for the subfield of fixed points; it is a subfield
of index $2$ except in the case $\E_i=\F=\F_\so$ (so that
$\b_i=0$).  

We have $\G_\E=\G_{\E_0}\times\prod_{i=1}^l\til\G_{\E_i}$ and
$\G_{\E_0}=\prod_{i\in\I_0}\G_{\E_i}$, where, for $i\in\I_0$, each
$\G_{\E_i}$ is the group of points of a unitary, symplectic or special
orthogonal group over $\E_{i,\so}$. (For each $i\in\I_0$, there is a
nondegenerate $\E_i/\E_{i,\so}$ $\e$-hermitian form $f_i$ on $\V^i$
such that the notions of lattice duality for $\Oo_{\E_i}$-lattices in
$\V^i$ given by $h|_{\V^i\times\V^i}$ and by $f_i$ coincide; then
$\G_{\E_i}$ is the group determined by this form.)

For $k\ge 0$, we write
$\P_k(\La_\oe)=\P_k(\La)\cap\G_\E=\til\P_k(\La_\oe)\cap\G$ and denote
by $\P^\so(\La_\oe)$ the inverse image in $\P(\La_\oe)=\P_0(\La_\oe)$
of the connected component of the reductive quotient
$\P(\La_\oe)/\P_1(\La_\oe)$. 

%%%%%%%%%%%%%%%%%%%%%%%%%%%%%%%%%%%%%%%%%%%%%%%%%%%%%%%%%%%%%%%%%%%%%%%%
\subsection{} 
The following two results are straightforward generalizations of
results from~\cite{S4}.

\begin{lemm}[{cf.~\cite[Proposition~3.4]{S4}}]
Let $[\La,n,0,\b]$ be a self-dual semisimple stratum in $\A$, with
associated splitting $\V=\bigoplus_{i\in\I}\V^i$. For
$0\le m\le n$, there is a self-dual semisimple stratum $[\La,n,m,\g]$
equivalent to $[\La,n,m,\b]$ such that $\g\in\Ll_\b^-$; in particular,
its associated splitting is a coarsening of $\bigoplus_{i\in\I}\V^i$.
\end{lemm}

Let $[\La,n,m,\b]$ be a self-dual semisimple stratum in $\A$ and, for
$i\in\I_+\cup\I_0$, let $s_i:\A^{ii}\to\B_i$ be a \emph{tame corestriction}
relative to $\E_i/\F$ (see~\cite[\S1.3]{BK} for the definition); for
$i\in\I_0$ we may and do assume $s_i$ commutes with the involution.

\begin{lemm}[{cf.~\cite[Lemma~3.5]{S4}}]
Let $[\La,n,m,\b]$ be a self-dual semisimple stratum in $\A$, with
associated splitting $\V=\bigoplus_{i\in\I}\V^i$. For
$i\in\I_+\cup\I_0$, let $b_i\in\PP_{-m}(\La)\cap\A^{ii}$ be such that
$[\La^i_{\Oo_{\E_i}},m,m-1,s_i(b_i)]$ is equivalent to a semisimple
stratum, and assume that $b_i\in\A_-$ for $i\in\I_0$. Put
$b_i=-\overline b_{-i}$, for $i\in\I_-$, and
$b=\sum_{i\in\I}b_i$. Then $[\La,n,m-1,\b+b]$ is equivalent to a
self-dual semisimple stratum, whose associated splitting is a
refinement of $\bigoplus_{i\in\I}\V^i$.
\end{lemm}

The point of these lemmas is that now all objects associated to a
self-dual semisimple stratum may be defined inductively with all
intermediate strata also self-dual semisimple. In particular, all the
objects will be stable under the involution $\s$.

%%%%%%%%%%%%%%%%%%%%%%%%%%%%%%%%%%%%%%%%%%%%%%%%%%%%%%%%%%%%%%%%%%%%%%%%
\subsection{}
A self-dual semisimple stratum $[\La,n,m,\b]$ is called {\it skew}
if its associated splitting $\V=\bigoplus_{i\in\I}\V^i$ is orthogonal;
equivalently, in the notation above, if $\I=\I_0$.

\begin{lemm} 
Let $[\La,n,0,\b']$ be a self-dual semisimple stratum in $\A$ and
suppose $[\La,n,m,\b']$ is equivalent to a self-dual semisimple stratum
$[\La,n,m,\b]$ with $\b\in\Ll_{\b'}$. Write
$\V=\bigoplus_{i\in\I}\V^i$ for the splitting associated to
$[\La,n,m,\b]$, which is a coarsening of that for $[\La,n,0,\b']$. 
\begin{enumerate}
\item For each $i\in\I_+\cup\I_0$, the derived stratum
  $[\La^i_{\Oo_{\E_i}},m,m-1,s_i(\b'_i-\b_i)]$ is either null or
  equivalent to a semisimple stratum.
\item Suppose $0<m\le n$ is minimal
  such that $[\La,n,m,\b']$ is equivalent to a skew semisimple stratum.
Then~$[\La,n,m,\b]$ is skew and there is an $i\in\I=\I_0$ such that the derived
stratum $[\La^i_{\Oo_{\E_i}},m,m-1,s_i(\b'_i-\b_i)]$ is $\G$-split.
\end{enumerate}
\end{lemm}

\begin{proof}(i) Write~$\V=\bigoplus_{j\in\I'}\V^j$ for the
splitting associated to~$[\La,n,0,\b']$ and~$\ee^j$ for the
associated idempotents; then, for each~$j\in\I'$,
there is a unique~$i\in\I$ such that~$\V^j\subseteq\V^i$. 
Now, applying~\cite[Theorem~2.4.1]{BK} to the simple
stratum~$[\La^j,n,m,\ee^j\b\ee^j]$ and the pure 
stratum~$[\La^j,n,m,\ee^j\b'\ee^j]$, we see 
that~$[\La^j_{\Oo_{\E_i}},m,m-1,\ee^j\(s_i(\b'_i-\b_i)\)\ee^j]$ is
either null or equivalent to a simple stratum. The result follows
since any direct sum of simple or null strata is equivalent to a
semisimple stratum.

(ii) If~$[\La^i_{\Oo_{\E_i}},m,m-1,s_i(\b'_i-\b_i)]$ is not~$\G$-split
then it is skew; thus, if
no~$[\La^i_{\Oo_{\E_i}},m,m-1,s_i(\b'_i-\b_i)]$ is $\G$-split then,
by~\cite[Lemma~3.5]{S4}, the stratum~$[\La,n,m-1,\b']$ is equivalent
to a skew semisimple stratum, contradicting the minimality of~$m$.
\end{proof}

%%%%%%%%%%%%%%%%%%%%%%%%%%%%%%%%%%%%%%%%%%%%%%%%%%%%%%%%%%%%%%%%%%%%%%%%
\subsection*{}{\bf{Self-dual semisimple characters and Heisenberg
extensions}}
%%%%%%%%%%%%%%%%%%%%%%%%%%%%%%%%%%%%%%%%%%%%%%%%%%%%%%%%%%%%%%%%%%%%%%%%
\subsection{}
Let~$[\La,n,0,\b]$ be a semisimple stratum in~$\A$. Associated to this
are certain orders~$\til\HH=\til\HH(\b,\La)$ and~$\til\JJ=\til\JJ(\b,\La)$ 
in~$\A$ (see~\cite[\S3.2]{S4}), along with compact groups with filtration
\[
\til\H=\til\H(\b,\La)=\til\HH\cap\til\P(\La),\qquad 
\til\H^n=\til\H^n(\b,\La)=\til\H\cap\til\P_n(\La),\hbox{ for }n\ge 1,
\]
and similarly for~$\til\J$. For each~$m\ge 0$ there is also a 
set~$\Cc(\La,m,\b)$ of \emph{semisimple characters} of the group~$\til\H^{m+1}$ 
(see~\cite[Definition~3.13]{S4}) with nice properties, some of which we 
recall in Lemma~\ref{lem:tth} below.

Recall that, given a representation~$\rho$ of a subgroup~$\til\K$ 
of~$\tG$ and~$g\in\tG$, the~\emph{$g$-intertwining space} of~$\rho$ is
\[
\I_g(\rho)=\I_g(\rho\mid\til\K)=\Hom_{\til\K\cap{}^g\til\K}(\rho,{}^g\rho),
\]
where~$^g\rho$ is the representation of~$^g\til\K=g\til\K g^{-1}$ given 
by $^g\rho(gkg^{-1}) = \rho(k)$, and the~\emph{$\tG$-intertwining} of~$\rho$ 
is
\[
\I_\tG(\rho) = \I_\tG(\rho\mid\til\K) = \{g\in\tG : \I_g(\rho) \ne\{0\}\}.
\]

\begin{lemm}[{\cite[Theorem~3.22,~Corollary~3.25]{S4}}]
\label{lem:tth}
Let~$\tth\in\Cc(\La,0,\b)$. Then
\begin{enumerate}
\item the intertwining of~$\tth$ is given 
by~$\I_{\tG}(\tth)=\til\J^1\tG_\E\til\J^1$;
\item there is a unique irreducible representation~$\til\eta$ 
of~$\til\J^1$ which contains~$\tth$; 
moreover,~$\I_{\tG}(\til\eta)=\til\J^1\tG_\E\til\J^1$.
\end{enumerate}
\end{lemm}

%%%%%%%%%%%%%%%%%%%%%%%%%%%%%%%%%%%%%%%%%%%%%%%%%%%%%%%%%%%%%%%%%%%%%%%%
\subsection{}
Now suppose~$[\La,n,0,\b]$ is a self-dual semisimple stratum and retain 
the notation of the previous paragraph. The associated 
orders and groups are invariant under the action of the involution~$\s$ 
and we put~$\H=\til\H\cap\G$ etc., as usual. The set~$\Cc_-(\La,m,\b)$ 
of \emph{self-dual semisimple characters} is the set of restrictions 
to~$\H^{m+1}$ of the semisimple characters~$\th\in\Cc(\La,m,\b)^\Si$; this 
can also be described in terms of the Glauberman correspondence 
(cf.~\cite[\S3.6]{S4}). The next lemma now follows exactly as 
in~\cite[Proposition~3.27,~Proposition~3.31]{S4}.

\begin{lemm}\label{lem:HeisenbergG}
Let~$\th\in\Cc_-(\La,0,\b)$. Then
\begin{enumerate}
\item the intertwining of~$\th$ is given 
by~$\I_{\G}(\th)=\J^1\G_\E\J^1$;
\item there is a unique irreducible representation~$\eta$ 
of~$\J^1$ which contains~$\th$; if~$\th=\tth|_{\H^1}$, 
for~$\tth\in\Cc(\La,0,\b)^\Si$ and~$\til\eta$ is the corresponding 
representation of~$\til\J^1$, then~$\eta$ is the Glauberman transfer 
of~$\til\eta$.
\end{enumerate}
\end{lemm}

%%%%%%%%%%%%%%%%%%%%%%%%%%%%%%%%%%%%%%%%%%%%%%%%%%%%%%%%%%%%%%%%%%%%%%%%
\subsection*{}{\bf{Transfer}}
%%%%%%%%%%%%%%%%%%%%%%%%%%%%%%%%%%%%%%%%%%%%%%%%%%%%%%%%%%%%%%%%%%%%%%%%
\subsection{} Let~$[\La,n,0,\b]$ and~$[\La',n',0,\b]$ be semisimple 
strata in~$\A$. Then (see~\cite[Proposition~3.26]{S4}) there is a 
canonical bijection (called the~\emph{transfer})
\[
\tau_{\La,\La',\b}:\Cc(\La,0,\b)\to\Cc(\La',0,\b)
\]
such that, for~$\tth\in\Cc(\La,0,\b)$, the 
character~$\tth':=\tau_{\La,\La',\b}(\tth)$ is the unique semisimple character 
in~$\Cc(\La',0,\b)$ such that~$\tG_\E\cap\I_\tG(\tth,\tth')\ne\emptyset$. 
Indeed,~$\tG_\E\subseteq\I_\tG(\tth,\tth')$. 

If the semisimple strata are self-dual then the 
bijection~$\tau_{\La,\La',\b}$ commutes with the involution 
(cf.~\cite[Proposition~3.32]{S4}) so induces a 
bijection~$\tau_{\La,\La',\b}:\Cc_-(\La,0,\b)\to\Cc_-(\La',0,\b)$.

Since, by Lemma~\ref{lem:HeisenbergG}, for each~$\th\in\Cc_-(\La,0,\b)$ 
there is a unique Heisenberg extension~$\eta$, we will also 
write~$\tau_{\La,\La',\b}(\eta)$ for the Heisenberg extension~$\eta'$ 
of~$\th':=\tau_{\La,\La',\b}(\th)$.

%%%%%%%%%%%%%%%%%%%%%%%%%%%%%%%%%%%%%%%%%%%%%%%%%%%%%%%%%%%%%%%%%%%%%%%%
\subsection{} Now suppose~$[\La,n,0,\b]$ and~$[\La',n',0,\b]$ are 
self-dual semisimple strata with the additional property 
that~$\AA(\La_\oe)\subseteq\AA(\La'_\oe)$. Let~$\th\in\Cc_-(\La,0,\b)$, 
denote by~$\eta$ the Heisenberg representation given by 
Lemma~\ref{lem:HeisenbergG}, and put~$\th'=\tau_{\La,\La',\b}(\th)$ 
and~$\eta'=\tau_{\La,\La',\b}(\eta)$. We form the 
group~$\J^1_{\La,\La'}=\P_1(\La_\oe)\J^1(\b,\La')$. As 
in~\cite[Propositions~3.7,~3.12,~Corollary~3.11]{S5} (see also~\cite[Proposition~1.2]{Bl}), we have:

\begin{prop}\label{prop:etaLL'}
There is a unique irreducible representation~$\eta_{\La,\La'}$ 
of~$\J^1_{\La,\La'}$ such that
\begin{enumerate}
\item $\eta_{\La,\La'}|_{\J^1(\b,\La')}=\eta'$, and
\item for any self-dual semisimple stratum~$[\La'',n'',0,\b]$ such 
that~$\AA(\La_{\oe})=\AA(\La''_\oe)$ and $\AA(\La'')\subseteq\AA(\La')$, we 
have that $\eta_{\La,\La'}$ and~$\tau_{\La,\La'',\b}(\eta)$ induce equivalent irreducible representations of~$\P_1(\La'')$.
\end{enumerate}
The intertwining of~$\eta_{\La,\La'}$ is given by
\[
\dim\I_g(\eta_{\La,\La'})\ =\ \begin{cases}
1 &\hbox{ if }g\in\J^1_{\La,\La'}\G^+_E\J^1_{\La,\La'}, \\
0 &\hbox{ otherwise.}
\end{cases}
\]
Moreover, if~$\AA(\La_\oe)$ is a minimal self-dual~$\oe$-order contained 
in~$\AA(\La'_\oe)$ then~$\eta_{\La,\La'}$ is the unique extension 
of~$\eta'$ to~$\J^1_{\La,\La'}$ which is intertwined by all of~$\G_\E$.
\end{prop}

%%%%%%%%%%%%%%%%%%%%%%%%%%%%%%%%%%%%%%%%%%%%%%%%%%%%%%%%%%%%%%%%%%%%%%%%
\subsection*{}{\bf{Standard $\b$-extensions}}
%%%%%%%%%%%%%%%%%%%%%%%%%%%%%%%%%%%%%%%%%%%%%%%%%%%%%%%%%%%%%%%%%%%%%%%%
\subsection{} We continue with the notation of the previous paragraph 
so~$\th\in\Cc_-(\La,0,\b)$ and~$\eta$ is the Heisenberg representation, 
while~$\th',\eta'$ are their transfers to the self-dual semisimple 
stratum~$[\La',n',0,\b]$, with~$\AA(\La_\oe)\subseteq\AA(\La'_\oe)$. 
We form the groups~$\J^+=\til\J(\b,\La)\cap\G^+$ 
and~$\J_{\La,\La'}^+=\P^+(\La_\oe)\J^1(\b,\La')$.

\begin{lemm}[{\cite[Lemma~4.3]{S5}}] 
In this situation, there is a canonical bijection~$\BB_{\La,\La'}$ 
from the set of extensions~$\k$ of~$\eta$ to~$\J^+$ to the set 
of extensions~$\k'$ of~$\eta'$ to~$\J_{\La,\La'}^+$. 

If~$\AA(\La)\subseteq\AA(\La')$ then~$\k'=\BB_{\La,\La'}(\k)$ is the unique 
extension of~$\eta'$ such that~$\k,\k'$ induce equivalent irreducible representations of~$\P^+(\La_\oe)\P_1(\La)$.
\end{lemm}

%%%%%%%%%%%%%%%%%%%%%%%%%%%%%%%%%%%%%%%%%%%%%%%%%%%%%%%%%%%%%%%%%%%%%%%%
\subsection{} For~$[\La,n,0,\b]$ a self-dual semisimple stratum, we define 
a related self-dual~$\oe$-lattice sequence~$\MM_\La$ as follows. Recall 
that we have the decomposition~$\V=\bigoplus_{i\in\I}\V^i$ 
and~$\I=\I_-\cup\I_0\cup\I_+$. For~$i\in\I$,~$r\in\ZZ$ and~$s=0,1$, we put
\[
\MM_\La^i(2r+s)=\begin{cases}
\p_{\E_i}^r\La^i(0)&\hbox{ if }i\in\I_+,\\
\p_{\E_i}^r\La^i(s)&\hbox{ if }i\in\I_0,\\
\p_{\E_i}^r\La^i(1)&\hbox{ if }i\in\I_-.
\end{cases}
\]
Then~$\MM_\La:=\bigoplus_{i\in\I}\MM_\La^i$ is a self-dual~$\oe$-lattice 
sequence on~$\V$ with the property that~$\AA(\MM_\La)\cap\B_\b$ is a 
maximal self-dual~$\oe$-order in~$\B_\b$.

Now we can define the notion of a standard~$\b$-extension. 

\begin{defi}[{\cite[Definition~4.5]{S5}}]
Let~$[\La,n,0,\b]$ be a self-dual semisimple stratum, 
let~$\th\in\Cc_-(\La,0,\b)$ and let~$\eta$ be the Heisenberg 
representation containing~$\th$.
\begin{enumerate}
\item Suppose~$\AA(\La_\oe)$ is a maximal self-dual~$\oe$-order in~$\B$. 
Then a representation~$\k$ of~$\J^+$ is called 
a~\emph{(standard) $\b$-extension of~$\eta$} if, for~$\La^\sm$ any 
self-dual~$\oe$-lattice sequence such that~$\AA(\La^\sm_\oe)$ is a 
minimal self-dual~$\oe$-order contained in~$\AA(\La_\oe)$,
it is an extension of the representation~$\eta_{\La^\sm,\La}$ of 
Proposition~\ref{prop:etaLL'}.
\item In general, a representation~$\k$ of~$\J^+$ is called 
a~\emph{standard $\b$-extension of~$\eta$} if there is 
a~$\b$-extension~$\k_\MM$ of~$\eta_\MM=\tau_{\La,\MM_\La,\b}(\eta)$ such 
that~$\BB_{\La,\MM_\La}(\k)=\k_\MM|_{\J^+_{\La,\MM_\La}}$. In this case
we say that~$\k_\MM$ is \emph{compatible with~$\k$}.
\end{enumerate}
\end{defi}

We will often say that~$\k$ is a \emph{standard~$\b$-extension of~$\th$}, 
since~$\eta$ is determined by~$\th$. We will also say that the 
restriction to~$\J$ (respectively~$\J^\so$) 
of a standard~$\b$-extension~$\k$ is a~\emph{standard $\b$-extension 
of~$\th$ to~$\J$} (respectively~$\J^\so$).

We also remark that~$\b$-extensions of a semisimple 
character~$\tth\in\Cc(\La,0,\b)$ for~$\tG$ may be defined in the same way. 
This generalizes the construction for simple characters and strict lattice 
sequences in~\cite[\S5.2]{BK}.

%%%%%%%%%%%%%%%%%%%%%%%%%%%%%%%%%%%%%%%%%%%%%%%%%%%%%%%%%%%%%%%%%%%%%%%%
\subsection*{}{\bf Iwahori decompositions}
%%%%%%%%%%%%%%%%%%%%%%%%%%%%%%%%%%%%%%%%%%%%%%%%%%%%%%%%%%%%%%%%%%%%%%%%
\subsection{}
Let~$[\La,n,0,\b]$ be a semisimple stratum in~$\A$ with associated 
splitting~$\V=\bigoplus_{i\in\I}\V^i$ and 
let~$\V=\bigoplus_{j=1}^m \W_j$ be a decomposition into subspaces which 
is~\emph{properly subordinate} to~$[\La,n,0,\b]$ in the sense 
of~\cite[Definition~5.1]{S5}: that is, each~$\W_j\cap\V^i$ is 
an~$\E_i$-subspace of~$\V^i$ 
and~$\W_j=\bigoplus_{i\in\I}\left(\W_j\cap\V^i\right)$, we have
\[
\La(r)=\bigoplus_{j=1}^m \left(\La(r)\cap\W_j\right), 
\quad\hbox{for all }r\in\ZZ,
\]
and, for each~$r\in\ZZ$ and~$i\in\I$, there is at most one~$j$ such that
\[
\left(\La(r)\cap\W_j\cap\V^i\right)\supsetneq 
\left(\La(r+1)\cap\W_j\cap\V^i\right).
\]
Denote by~$\til\M$ the Levi subgroup of~$\tG$ which is the stabilizer 
of the decomposition~$\V=\bigoplus_{j=1}^m \W_j$ and let~$\til\P$ be 
any parabolic subgroup with Levi component~$\til\M$ and unipotent 
radical~$\til\U$.

By~\cite[Proposition~5.2]{S5}, the groups~$\til\J$,~$\til\J^1$ 
and~$\til\H^1$ have Iwahori decompositions with respect 
to~$(\til\M,\til\P)$ and we put
\[
\til\H^1_{\til\P}=\til\H^1\left(\til\J^1\cap\til\U\right),\quad 
\til\J^1_{\til\P}=\til\H^1\left(\til\J^1\cap\til\P\right),\quad
\hbox{ and }\quad 
\til\J_{\til\P}=\til\H^1\left(\til\J\cap\til\P\right).
\]
For~$\tth\in\Cc(\La,0,\b)$ we define the character~$\tth_{\til\P}$ 
of~$\til\H^1_{\til\P}$ by
\[
\tth_{\til\P}(hj)=\tth(h),\quad\hbox{for }h\in\til\H^1,\ 
j\in\til\J^1\cap\til\U.
\]
This is well-defined.

\begin{lemm}[{\cite[Corollary~5.7,~Lemma~5.8]{S5}}]
Let~$\tth\in\Cc(\La,0,\b)$ and let~$\til\eta$ be the corresponding
representation of~$\til\J^1$. Then
\begin{enumerate}
\item the intertwining of~$\tth_{\til\P}$ is given 
by~$\I_{\tG}(\tth_{\til\P})=\til\J_{\til\P}^1\tG_\E\til\J_{\til\P}^1$;
\item there is a unique irreducible representation~$\til\eta_{\til\P}$ 
of~$\til\J^1_{\til\P}$ which contains~$\tth_{\til\P}$; 
moreover,~$\I_{\tG}(\til\eta_{\til\P})=
\til\J^1_{\til\P}\tG_\E\til\J^1_{\til\P}$ 
and~$\til\eta\simeq\Ind_{\til\J_{\til\P}^1}^{\til\J^1}\til\eta_{\til\P}$.
\end{enumerate}
\end{lemm}

%%%%%%%%%%%%%%%%%%%%%%%%%%%%%%%%%%%%%%%%%%%%%%%%%%%%%%%%%%%%%%%%%%%%%%%%
\subsection{}
Now suppose~$[\La,n,0,\b]$ is a self-dual semisimple stratum 
and~$\V=\bigoplus_{j=-m}^m\W_j$ is a properly subordinate 
\emph{self-dual} decomposition, that 
is, the orthogonal complement of~$\W_j$ is~$\bigoplus_{k\ne-j}\W_k$, 
for each~$j$. (We allow the possibility that~$\W_0=\{0\}$.) We use 
the notation of the previous paragraph and put~$\M=\til\M\cap\G$, a 
Levi subgroup of~$\G$, and, choosing~$\til\P$ to be a~$\s$-stable 
parabolic subgroup of~$\tG$, put~$\P=\til\P\cap\G=\M\U$, a parabolic 
subgroup of~$\G$. Then~$\H^1$ has an Iwahori decomposition with respect 
to~$(\M,\P)$, while~$\til\H^1_{\til\P}$ is stable under the involution, 
and we 
put~$\H^1_\P=\til\H^1_{\til\P}\cap\G=\H^1\left(\J^1\cap\U\right)$. 
Similarly, we have~$\J^1_\P$,~$\J_\P$ and~$\J_\P^+$, as well 
as~$\J_\P^\so=\H^1\left(\J^\so\cap\P\right)$.

For~$\th\in\Cc_-(\La,0,\b)$, define the character~$\th_\P$ 
of~$\H^1_\P$ by
\[
\th_\P(hj)=\th(h),\quad\hbox{for }h\in\H^1,\ 
j\in\J^1\cap\U;
\]
thus, if~$\th=\tth|_{\H^1}$ for some~$\tth\in\Cc(\La,0,\b)^\Si$, 
then~$\th_\P=\tth_{\til\P}|_{\H^1_\P}$. Exactly as 
in~\cite[Lemma~5.12]{S5}, we get:

\begin{lemm}\label{lem:thetaP}
Let~$\th\in\Cc_-(\La,0,\b)$ and let~$\eta$ be the corresponding
representation of~$\J^1$. Then
\begin{enumerate}
\item the intertwining of~$\th_\P$ is given 
by~$\I_{\G}(\th_\P)=\J_\P^1\G_\E\J_\P^1$;
\item there is a unique irreducible representation~$\eta_\P$ 
of~$\J_\P^1$ which contains~$\th_\P$; if~$\th=\tth|_{\H^1}$, 
for~$\tth\in\Cc(\La,0,\b)^\Si$ and~$\til\eta$ is the corresponding 
representation of~$\til\J^1$, then~$\eta_\P$ is the Glauberman transfer 
of~$\til\eta_{\til\P}$;
\item with~$\eta_\P$ as in~(ii), we 
have~$\eta\simeq\Ind_{\J_\P^1}^{\J^1}\eta_\P$ and
\[
\dim\I_g(\eta_\P)\ =\ \begin{cases}
1 &\hbox{ if }g\in\J_\P^1\G_E^+\J_\P^1, \\
0 &\hbox{ otherwise.}
\end{cases}
\]
\end{enumerate}
\end{lemm}

%%%%%%%%%%%%%%%%%%%%%%%%%%%%%%%%%%%%%%%%%%%%%%%%%%%%%%%%%%%%%%%%%%%%%%%%
\subsection{} We continue with the notation of the previous paragraph. 
Let~$\k$ be a standard~$\b$-extension of~$\eta$ to~$\J^+$. We form 
the natural representation~$\k_\P$ of~$\J_\P^+$ on the space 
of~$\J\cap\U$-fixed vectors in~$\k$; then~$\k_\P$ is an extension 
of~$\eta_\P$ and~$\Ind_{\J_P^+}^{\J^+}\k_\P\simeq\k$. Similar results 
apply to the restriction of~$\k_\P$ to~$\J_\P$ and to~$\J^\so_\P$.

We can also make the same construction for a~$\b$-extension~$\til\k$ of 
a semisimple character~$\tth$ for~$\tG$, thus obtaining a 
representation~$\til\k_{\til\P}$ of~$\til\J_{\til\P}$.

%%%%%%%%%%%%%%%%%%%%%%%%%%%%%%%%%%%%%%%%%%%%%%%%%%%%%%%%%%%%%%%%%%%%%%%%
\subsection{} Suppose~$[\La,n,0,\b]$ is a self-dual semisimple stratum, 
with associated Levi subgroup~$\L=\L_\b$ as in paragraph~\ref{S.semigroups}, 
which we identify with~$\G^0\times\prod_{i=1}^l\tG^i$. 
Note that the associated decomposition~$\V=\bigoplus_{i=-l}^l\V^i$ is 
properly subordinate to the stratum. Let~$\Q$ be a parabolic subgroup 
of~$\G$ with Levi component~$\L$. We write~$\H^1_\L=\H^1_\Q\cap\L=\H^1\cap\L$; 
then~$\H^1_\L=\H^1(\b_0,\La_0)\times\prod_{i=1}^l\til\H(\b_i,\La_i)$. 
Similarly we have~$\J^1_\L$, etc.

For~$\th\in\Cc_-(\La,0,\b)$ a semisimple character we 
put~$\th_\L=\th|_{\H^1_\L}$. 
Then~$\th_\L=\th_0\otimes\bigotimes_{i=1}^l\tth_i$, 
with~$\th_0$ a skew %self-dual? 
semisimple character in~$\Cc_-(\La_0,0,\b_0)$, 
and~$\tth_i$ a \emph{simple} character in~$\Cc(\La_i,0,2\b_i)$. 

By a~\emph{standard~$\b$-extension of~$\th_\L$}, we mean a 
representation~$\k_\L$ of~$\J^+_\L$ (or~$\J_\L$,~$\J^\so_\L$) of the 
form~$\k_\L=\k_0\otimes\bigotimes_{i=1}^l\til\k_i$, with~$\k_0$ a 
standard~$\b_0$-extension of~$\th_0$ and~$\til\k_i$ 
a~$2\b_i$-extension of~$\tth_i$.

%%%%%%%%%%%%%%%%%%%%%%%%%%%%%%%%%%%%%%%%%%%%%%%%%%%%%%%%%%%%%%%%%%%%%%%%
\subsection{} We continue with notation of the previous paragraph.

Let~$\V=\bigoplus_{j=-m}^m\W_j$ be another self-dual decomposition properly 
subordinate to~$[\La,n,0,\b]$ and~$\M$ the Levi subgroup of~$\G$ 
stabilizing the decomposition. We suppose also that~$\M\subseteq\L$ and 
let~$\P=\M\U\subseteq\Q$ be a parabolic subgroup of~$\G$ with Levi 
component~$\M$. Then~$\P\cap\L=\M(\U\cap\L)$ is a parabolic subgroup of~$\L$ 
with Levi component~$\M$.

Let~$\th\in\Cc_-(\La,0,\b)$ be a semisimple character and let~$\k$ be a 
standard~$\b$-extension to~$\J$. We form the representation~$\k_\P$ 
of~$\J_\P$ as above, and also the representation~$\k_\Q$ of~$\J_\Q$. 
Note that, since~$\M\subseteq\L$, we have~$\J_\P\subseteq\J_\Q$, 
and~$\k_\P$ can be viewed as the natural representation on 
the~$\J\cap\U$-fixed vectors in~$\k_\Q$. 

We also have~$\J_\Q\cap\L=\J\cap\L$, since~$\G_\E\subseteq\L$,
and we can consider the natural representation of~$\J_\P\cap\L$ on 
the~$\J\cap\L\cap\U$-fixed vectors in~$\k_\Q|_{\J\cap\L}$. This is 
naturally isomorphic to the restriction~$\k_{\P}|_{\J_\P\cap\L}$. We 
will need the following compatibility result.

\begin{prop}\label{prop:kappaL}
In the situation above, the restriction
$\k_{\P}|_{\J_\P\cap\L}$ takes the form $\k'_{\P\cap\L}$, where
$\k':=\k_\Q|_{\J\cap\L}$ is a standard $\b$-extension of~$\th_\L$ 
to~$\J_\L=\J\cap\L$.
\end{prop}

\begin{proof}
We need to check that~$\k':=\k_\Q|_{\J\cap\L}$ is a 
standard~$\b$-extension. If~$\AA(\La_{\oe})$ is a maximal self-dual 
order in~$\B$ (in which case~$\L=\M$) then this follows 
from~\cite[Proposition~6.3]{S5}. 

For the general case, denote by~$\k_\MM$ 
the unique~$\b$-extension of~$\J_\MM=\J(\b,\MM_\La)$ compatible with~$\k$; 
then~$\k_{\MM,\Q}|_{\J_\MM\cap\L}$ is a (standard)~$\b$-extension by the 
previous case, and~$\k_\Q|_{\J\cap\L}$ is compatible 
with~$\k_{\MM,\Q}|_{\J_\MM\cap\L}$, by~\cite[Proposition~1.17]{Bl}. 
Thus~$\k'$ is indeed a standard~$\b$-extension.
\end{proof}

%%%%%%%%%%%%%%%%%%%%%%%%%%%%%%%%%%%%%%%%%%%%%%%%%%%%%%%%%%%%%%%%%%%%%%%%
%%%%%%%%%%%%%%%%%%%%%%%%%%%%%%%%%%%%%%%%%%%%%%%%%%%%%%%%%%%%%%%%%%%%%%%%
%%%%%%%%%%%%%%%%%%%%%%%%%%%%%%%%%%%%%%%%%%%%%%%%%%%%%%%%%%%%%%%%%%%%%%%%
%%%%%%%%%%%%%%%%%%%%%%%%%%%%%%%%%%%%%%%%%%%%%%%%%%%%%%%%%%%%%%%%%%%%%%%%

\section{Cuspidal types}\label{S.cuspidal}

In this section we recall the notions of cuspidal types from~\cite{BK,S5}, 
correcting along the way a mistake in the definition in~\cite{S5} pointed
out by Laure Blasco and Corinne Blondel. 
%Maybe also simple types for $\GL_N$ from~\cite{BK,SS6}, if necessary.

%%%%%%%%%%%%%%%%%%%%%%%%%%%%%%%%%%%%%%%%%%%%%%%%%%%%%%%%%%%%%%%%%%%%%%%%
\subsection{}
We recall from~\cite{BK} the definition of a \emph{simple type} and
of a \emph{maximal simple type} for~$\tG$; we call the latter a 
\emph{cuspidal type}. The generalizations to the case of lattice sequences 
come from~\cite{SS6}.

\begin{defi}%\label{def:simpletype}
A \emph{simple type} for~$\tG$ is a pair~$(\til\J,\til\l)$, 
where~$\til\J=\til\J(\b,\La)$ for some simple stratum~$[\La,n,0,\b]$ 
such that
\begin{itemize}
\item[$\bullet$] $\til\P(\La_{\oe})/\til\P_1(\La_{\oe})\simeq\GL_f(k_\E)^e$, for some 
positive integers~$f,e$,
\end{itemize}
and~$\til\l=\til\k\otimes\til\tau$, for~$\til\k$ a~$\b$-extension of some 
simple character~$\tth\in\Cc(\La,0,\b)$
and~$\til\tau$ the inflation of an irreducible cuspidal 
representation~$\til\tau_0^{\otimes e}$ 
of~$\til\J/\til\J^1\simeq\GL_f(k_\E)^e$.

A \emph{cuspidal type} for~$\tG$ is a simple type for 
which~$\til\P(\La_{\oe})$ is a maximal parahoric subgroup of~$\tG_\E$; 
that is,~$e=1$ in the notation above.
\end{defi}

Every irreducible cuspidal representation~$\til\pi$ of~$\tG$ contains a 
cuspidal type~$(\til\J,\til\l)$. Then~$\til\pi$ is irreducibly compactly 
induced from a representation of~$\E^\times\til\J$ containing~$\til\l$ 
and the cuspidal type~$(\til\J,\til\l)$ is a~$[\tG,\til\pi]_\tG$-type.

The following proposition can be extracted from the results 
in~\cite[\S\S7--8]{BK} (see also~\cite[Proposition~5.15, Corollaire~5.20]{SS4}).

\begin{prop}\label{prop:cuspGL}
Let~$[\La,n,0,\b]$ be a simple stratum,~$\tth\in\Cc(\La,0,\b)$ 
a simple character, and~$\til\k$ a $\b$-extension. Let~$\til\tau$ be 
(the inflation to~$\til\J$ of) an irreducible representation 
of~$\til\P(\La_\oe)/\til\P_1(\La_\oe)$. Suppose a cuspidal 
representation~$\til\pi$ of~$\tG$ contains~$\tth$ 
and~$\til\k\otimes\til\tau$. Then~$\til\P(\La_\oe)$ is a maximal 
parahoric subgroup of~$\tG_\E$, and~$\til\tau$ is cuspidal; that 
is,~$(\til\J,\til\k\otimes\til\tau)$ is a cuspidal type.
\end{prop}

\subsection{}
Now we recall from~\cite{S5} the (corrected) definition of a 
maximal simple type for~$\G$, which we again call a cuspidal type. Recall 
that we have assumed that~$\G$ is not itself a split two-dimensional special 
orthogonal group; thus its centre is compact.

\begin{defi}\label{def:cuspidaltype}
A \emph{cuspidal type} for~$\G$ is a pair~$(\J,\l)$, 
where~$\J=\J(\b,\La)$ for some skew semisimple 
stratum~$[\La,n,0,\b]$ such that
\begin{itemize}
\item[$\bullet$] $\G_\E$ has compact centre and
\item[$\bullet$] $\P^\so(\La_{\oe})$ a maximal parahoric subgroup 
of~$\G_\E$, 
\end{itemize}
and~$\l=\k\otimes\tau$, for~$\k$ a~$\b$-extension and~$\tau$ the inflation 
of an irreducible cuspidal representation 
of~$\J/\J^1\simeq\P(\La_{\oe})/\P_1(\La_{\oe})$.
\end{defi}

\begin{rema}
We thank Laure Blasco and Corinne Blondel for pointing out the problem 
with the definition in~\cite[Definition~6.17]{S5}. There, the two 
conditions on the stratum~$[\La,n,0,\b]$ in 
Definition~\ref{def:cuspidaltype} are replaced  by the (insufficient) 
condition that~$\AA(\La_{\oe})$ be a maximal self-dual~$\oe$-order 
in~$\B$. 

Firstly, this is not enough to guarantee that~$\P^\so(\La_{\oe})$ 
be a maximal parahoric subgroup of~$\G_\E$: for example, if~$\G_\E$ is a 
quasi-split ramified unitary group in~$2$ variables then, for one of the 
two (up to conjugacy) maximal self-dual~$\oe$-orders, the corresponding 
parahoric subgroup is an Iwahori subgroup, so not maximal. 

Secondly, even if~$\P^\so(\La_{\oe})$ is a maximal parahoric subgroup, it 
can still happen that its normalizer in~$\G_\E$ is not compact: this 
happens precisely when~$\G_\E$ has a factor isomorphic to the 
split torus~$\SO(1,1)(\F)$, which can only happen when~$\G$ is an 
even-dimensional orthogonal group and~$\b_i=0$, $\dim_\F\V^i=2$, 
for some~$i\in\I_0$. The condition that~$\G_\E$ have compact centre 
rules out exactly this possibility.

In particular, with the definition of cuspidal type~$(\J,\l)$ given here, 
the proof of~\cite[Proposition~6.18]{S5} is valid, and~$\cInd_\J^\G \l$ 
is an irreducible cuspidal representation of~$\G$.
\end{rema}

%%%%%%%%%%%%%%%%%%%%%%%%%%%%%%%%%%%%%%%%%%%%%%%%%%%%%%%%%%%%%%%%%%%%%%%%
\subsection{}%\label{sub:cuspproof}
In this paragraph we will indicate the minor changes that must be made 
to~\cite[\S7.2]{S5} in order to correct the proof of  the main result 
there~\cite[Theorem~7.14]{S5}: every irreducible cuspidal 
representation of~$\G$ contains a cuspidal type. This paragraph should 
be read alongside that paper and we will make free use of notations 
from there.

Suppose~$\pi$ is an irreducible representation of~$\G$ and suppose 
that there is a pair~$([\La,n,0,\b],\th)$, consisting of a skew semisimple 
stratum~$[\La,n,0,\b]$ and a semisimple character~$\th\in\Cc_-(\La,0,\b)$, 
such that~$\pi$ contains~$\th$. Suppose moreover that, for fixed~$\b$, we 
have chosen a pair for which the parahoric subgroup~$\P^\so(\La_\oe)$ 
is \emph{minimal amongst such pairs}. 
If~$\k$ is a standard $\b$-extension then~$\pi$ also 
contains a representation~$\vartheta=\k\otimes\rho$ of~$\J^\so$, 
for~$\rho$ an irreducible representation 
of~$\J^\so/\J^1\simeq\P^\so(\La_\oe)/\P_1(\La_\oe)$. 
By~\cite[Lemma~7.4]{S5}, the minimality of~$\P^\so(\La_\oe)$ implies 
that the representation~$\rho$ is cuspidal.

\emph{We suppose that either the parahoric subgroup~$\P^\so(\La_\oe)$ is 
not maximal in~$\G_\E$ or~$\G_\E$ does not have compact centre} and will 
find a non-zero Jacquet module. (This assumption takes the place of 
hypothesis~(H) in~\cite[\S7.2]{S5}.) Most of~\cite[\S7.2]{S5} now goes through 
essentially unchanged, with two small changes in the cases called~(i) and~(ii) 
in~\S7.2.2 (page~350). 

In case~(i), the change happens when the element~$p$~\emph{cannot} be
chosen to normalize the representation~$\rho$, interpreted as a
representation of~$\P^\so(\La_{\oe})$ trivial
on~$\P_1(\La_{\oe})$. (Note that~$p\in\P^+(\La_{\oe})$ so it does
normalize the group~$\P^\so(\La_{\oe})$.) In this
case,~$N_\La(\rho)\subseteq\M'$, where~$\M'$ is the Levi subgroup
of~\emph{loc. cit.} (Note, however, that this would not be the case if
we were working in the non-connected group~$\G^+$, rather than~$\G$.)
Thus, by~\cite[Corollary~6.16]{S5}, we
have~$\I_\G(\vartheta_\P)\subseteq\J_\P^\so\M'\J_\P^\so$ and, as in
the proof of~\cite[Proposition~7.10]{S5}, $(\J_\P^\so,\vartheta_\P)$
is a cover
of~$(\J_\P^\so\cap\M',\vartheta_\P|_{\J^\so_\P\cap\M'})$. (See also
Lemma~\ref{lem:iiia} below.)

In case~(ii), the change happens when~$m=1$. In this 
case~$\P^\so(\La_\oe)$ is a maximal parahoric subgroup but~$\G_\E$ does 
not have compact centre; indeed~$\G_{\E_1}\simeq\SO(1,1)(\F)$ and
we have~$\G_\E\subseteq\M'$. As in case~(ii) above, we get 
that~$\I_\G(\vartheta_\P)\subseteq\J_\P^\so\M'\J_\P^\so$ 
and~$(\J_\P^\so,\vartheta_\P)$ is a 
cover of~$(\J_\P^\so\cap\M',\vartheta_\P|_{\J^\so_\P\cap\M'})$. 

In particular, in all cases, the representation~$\pi$ 
containing~$\vartheta$ cannot be cuspidal. Thus we have the following 
analogue of Proposition~\ref{prop:cuspGL}.

\begin{prop}\label{prop:cuspcusp}
Let~$[\La,n,0,\b]$ be a skew semisimple stratum,~$\th\in\Cc_-(\La,0,\b)$ 
a semisimple character, and~$\k$ a standard $\b$-extension. Let~$\tau$ be 
(the inflation to~$\J$ of) an irreducible representation 
of~$\P(\La_\oe)/\P_1(\La_\oe)$. Suppose a cuspidal 
representation~$\pi$ of~$\G$ contains~$\th$ and~$\k\otimes\tau$. 
Then~$\G_\E$ has compact centre, $\P^\so(\La_\oe)$ is a maximal 
parahoric subgroup of~$\G_\E$, and~$\tau$ is cuspidal; that 
is,~$(\J,\k\otimes\tau)$ is a cuspidal type.
\end{prop}

\begin{proof}
We have just proved the first two assertions, while the third follows
from~\cite[Lemma~7.4]{S5}.
\end{proof}

In particular, since by~\cite[Theorem~5.1]{S4} every irreducible 
cuspidal representation of~$\G$ does contain a semisimple character, 
and hence a representation of~$\J$ of the form~$\k\otimes\tau$,
this also proves~\cite[Theorem~7.14]{S5}.

%%%%%%%%%%%%%%%%%%%%%%%%%%%%%%%%%%%%%%%%%%%%%%%%%%%%%%%%%%%%%%%%%%%%%%%%
%%%%%%%%%%%%%%%%%%%%%%%%%%%%%%%%%%%%%%%%%%%%%%%%%%%%%%%%%%%%%%%%%%%%%%%%
%%%%%%%%%%%%%%%%%%%%%%%%%%%%%%%%%%%%%%%%%%%%%%%%%%%%%%%%%%%%%%%%%%%%%%%%
%%%%%%%%%%%%%%%%%%%%%%%%%%%%%%%%%%%%%%%%%%%%%%%%%%%%%%%%%%%%%%%%%%%%%%%%

\section{Semisimple types}\label{S.covers}

In this section we will prove Theorem~\ref{thm:main} of the
introduction, explaining how to construct a type for each Bernstein
component, via the theory of covers.

%%%%%%%%%%%%%%%%%%%%%%%%%%%%%%%%%%%%%%%%%%%%%%%%%%%%%%%%%%%%%%%%%%%%%%%%
\subsection{} 
We suppose given a Levi subgroup $\M$ of $\G$, which is the stabilizer
of the self-dual decomposition
\begin{equation}\label{eqn:decomp}
\V\ =\ \W_{-m}\oplus \cdots \oplus\W_{m};
\end{equation}
thus, putting $\tG_j=\Aut_\F(\W_j)$ and $\G_0=\Aut_\F(\W_0)\cap\G$, we have
$\M=\G_0\times\prod_{j=1}^m\tG_j$. Let $\tau$ be a cuspidal
irreducible representation of $\M$, which we 
write~$\tau=\tau_0\otimes\bigotimes_{j=1}^m \ttau_j$. 

Also let $\Mm$ denote the stabilizer of the
decomposition~\eqref{eqn:decomp} in $\A$; thus~$\Mm_-$ is the Lie algebra of
$\M$. For $-m\le j\le m$, we denote by $\ee_j$ the idempotent given by
projection onto $\W_j$.

For each $j>0$, let $[\La_j,n_j,0,\b_j]$ be a simple stratum in $\A_j$
and let~$\til\th_j\in\Cc(\La_j,0,\b_j)$ be such that $\ttau_j$ contains
$\til\th_j$; let also $[\La_0,n_0,0,\b_0]$ be a skew semisimple stratum in $\A_0$
and let~$\th_0\in\Cc_-(\La_0,0,\b_0)$ be such that $\tau_0$ contains
$\th_0$.

\begin{prop}[{\cite[Proposition~8.4]{Dat}}]\label{prop:dat} 
There are a self-dual semisimple stratum $[\La,n,0,\b]$ with
$\b\in\Mm$, and a self-dual semisimple character $\th$ of
$\H^1=\H^1(\b,\La)$ such that: 
\begin{enumerate}
\item The decomposition~\eqref{eqn:decomp} is properly subordinate to
$[\La,n,0,\b]$;
\item $\H^1(\b,\La)\cap\M=\H^1(\b_0,\La_0)\times
\prod_{j=1}^m\til\H^1(\b_j,\La_j)$; and
\item $\th|_{\H^1(\b,\La)\cap\M}=\th_0\otimes\bigotimes_{j=1}^m\til\th_j$. 
\end{enumerate}
\end{prop}

\begin{proof} (ii) and (iii) are given by~\cite[Proposition~8.4]{Dat}, and (i) by the comments following its statement. 
\end{proof}

For $j\ne 0$, we note that (iii) implies that $\til\th_j$ is a simple
character for the simple stratum $[\La_j,n_j,0,2\ee_j\b\ee_j]$; likewise,
$\th_0$ is a skew semisimple character for $[\La_0,n_0,0,\ee_0\b\ee_0]$.
Thus we may, and do, assume that $\b_j=2\ee_j\b\ee_j$, for $j>0$, and
$\b_0=\ee_0\b\ee_0$. Similarly, we may and do assume that the lattice
sequence $\La_j$ is equal to $\La\cap\W_j$.

\begin{rema} The property in Proposition~\ref{prop:dat} 
that~$[\La,n,0,\b]$ is semisimple is strictly stronger than the 
property that each stratum $[\La_j,n_j,0,\ee_j\b\ee_j]$ is (semi)simple. 
In general, the direct sum of (semi)simple strata need not be semisimple.
\end{rema}

Let $\V=\bigoplus_{i=-l}^l \V^i$ be the self-dual decomposition associated
to the stratum $[\La,n,0,\b]$ and let $\L=\L_\b$ be the $\G$-stabilizer of this
decomposition. Since $\b\in\Mm$, this is a coarsening of the
decomposition~\eqref{eqn:decomp}: that is, each $\V^i$ is a sum of
certain $\W_j$, with $\W_0\subseteq\V^0$, so that $\L\supseteq\M$. 

We will abbreviate $\H^1=\H^1(\b,\La)$, and similarly
$\J^1,\J^\so,\J$. By Proposition~\ref{prop:dat}(i), all these groups
have Iwahori decompositions with respect to $(\M,\P)$, for any
parabolic subgroup $\P=\M\U$ with Levi component $\M$; thus we may
form the groups $\H^1_\P,\J^1_\P,\J^\so_\P,\J_\P$ as 
in~\S\ref{S.characters}. 

Write $\G_\E$ for the centralizer of $\b$ in $\G$, so $\G_\E\subseteq\L$. 
We note that, by Propositions~\ref{prop:cuspGL}~and~\ref{prop:cuspcusp}, 
the group
$\J^\so\cap\G_\E\cap\M$ is a maximal parahoric subgroup of $\G_\E\cap\M$. In
particular, the decomposition~\eqref{eqn:decomp} is \emph{exactly
subordinate} to $[\La,n,0,\b]$, in the language
of~\cite[Definition~6.5]{S5}. (In fact, the definition of exactly subordinate in~\emph{loc.\ cit.} should have required that~$\P^\so(\La_\oe)\cap\M$ be a maximal parahoric subgroup of $\G_\E\cap\M$.)

Let $\eta$ be the unique irreducible representation of $\J^1$
containing $\th$, and choose a standard $\b$-extension $\k$ of $\th$.
Denote by $\k_\P$ the natural representation of $\J_\P$ on the
$(\J\cap\U)$-fixed vectors in $\k$, by $\eta_\P$ its restriction to
$\J_\P^1$, and by $\th_\P$ the character of $\H^1_\P$ which extends
$\th$ and is trivial on $\J^1\cap\U$.

Since the decomposition~\eqref{eqn:decomp} is exactly subordinate to 
$[\La,n,0,\b]$, by~\cite[Proposition~6.3]{S5} the 
restriction~$\k_\M=\k_\P|_{\J\cap\M}$
is a standard $\b$-extension of $\eta_\M=\eta_\P|_{\J^1\cap\M}$, which is
itself the unique irreducible representation of $\J^1\cap\M$
containing $\th_\M=\th|_{\H^1\cap\M}$; this means that
$\k_\M=\k_0\otimes\bigotimes_{j=1}^m\til\k_j$, where $\til\k_j$ is a
$\b_j$-extension containing $\tth_j$ and $\k_0$ is a standard
$\b_0$-extension containing $\th_0$. 

Since $\tau$ contains $\th_\M$, it also contains $\eta_\M$, and hence
some representation of $\J^\so\cap\M$ of the form
$\l^\so_\M=\k_\M\otimes\rho^\so_\M$, with $\rho^\so_\M$ the inflation to
$\J^\so\cap\M$ of an irreducible representation of the connected
reductive group $\P^\so(\La_{\oe})/\P_1(\La_{\oe})$. Moreover, by
Propositions~\ref{prop:cuspGL}~and~\ref{prop:cuspcusp}, the 
representation $\rho^\so_\M$ is necessarily cuspidal. We write
$\rho^\so_\M=\rho^\so_0\otimes\bigotimes_{j=1}^m\til\rho_j$, where
$\til\rho_j$ is a cuspidal representation of
$\til\P(\La_{j,\fo_{\E_j}})/\til\P_1(\La_{j,\fo_{\E_j}})$, for $j\ge 1$, and
$\rho^\so_0$ is a cuspidal representation of
$\P^\so(\La_{0,\fo_{\E_0}})/\P_1(\La_{0,\fo_{\E_0}})$.

Now we have an isomorphism
$\J^\so_\P/\J^1_\P\simeq(\J^\so\cap\M)/(\J^1\cap\M)$ so we can also
regard $\rho^\so_\M$ as a representation of $\J^\so_\P$ by inflation. Thus
we can form the representation $\l^\so_\P=\k_\P\otimes\rho^\so_\M$ of
$\J_\P^\so$. The main result is then:

\begin{theo}\label{thm:cover}
The pair $(\J_\P^\so,\l^\so_\P)$ is a cover of $(\J_\P^\so\cap\M,\l^\so_\M)$.
\end{theo}

\begin{rema}
Certainly, the pair $(\J_\P^\so,\l^\so_\P)$ is a decomposed pair above $(\J_\P^\so\cap\M,\l^\so_\M)$, in the sense of~\cite[Definition~6.1]{BK1}. Moreover, putting~$\l^\so=\Ind_{\J_\P^\so}^{\J^\so}\l_\P^\so=\k\otimes\l_\M^\so$ (with~$\l_\M^\so$ regarded as a representation of~$\J^\so$ trivial on~$\J^1$), we have a support-preserving isomorphism
\[
\Hh(\G,\l_\P^\so)\simeq\Hh(\G,\l^\so),
\]
as in~\cite[Lemma~6.1]{S5}. In particular, the condition on the Hecke 
algebra which needs to be checked to prove that~$(\J_\P^\so,\l^\so_\P)$ 
is a cover is independent of the choice of parabolic subgroup~$\P$ with 
Levi component~$\M$. Thus we can, and will, change our choice of~$\P$ 
where necessary.
\end{rema}

The proof of Theorem~\ref{thm:cover} will occupy the next few subsections. 
Let us see how this implies Theorem~\ref{thm:main} of the introduction. 

\begin{proof}[Proof of Theorem~\ref{thm:main}]
Since~$\tau$ contains~$\l^\so_\M$, it contains some irreducible
representation~$\l_\M$ of~$\J\cap\M=\J_\P\cap\M$ which contains~$\l^\so_\M$; 
more precisely, we can write~$\l_\M=\k_\M\otimes\rho_\M$, with~$\rho_\M$ 
the inflation of an irreducible representation 
of~$\P(\La_{\oe})/\P_1(\La_{\oe})$ which contains~$\rho_\M^\so$. 
Thus~$(\J_\P\cap\M,\l_\M)$ is a cuspidal type in~$\M$, which is 
an~$[\M,\tau]_\M$-type.

We put $\l_\P=\k_\P\otimes\rho_\M$, so that
$\l_\P|_{\J_\P\cap\M}=\l_\M$. Then certainly $(\J_\P,\l_\P)$ is a 
decomposed pair above $(\J_\P\cap\M,\l_\M)$, while $(\J_\P^\so,\l^\so_\P)$ 
is a cover of $(\J_\P^\so\cap\M,\l^\so_\M)$, by Theorem~\ref{thm:cover}. 
Thus, by~\cite[Lemma~3.9]{M3}, $(\J_\P,\l_\P)$ is also a cover of 
$(\J_\P\cap\M,\l_\M)$. Since $(\J_\P\cap\M,\l_\M)$ is an 
$[\M,\tau]_\M$-type, 
we conclude from~\cite[Theorem~8.3]{BK1} that $(\J_\P,\l_\P)$ is an 
$[\M,\tau]_\G$-type. Since the pair $(\M,\tau)$ was arbitrary, we 
have a type for every Bernstein component.
\end{proof}

%%%%%%%%%%%%%%%%%%%%%%%%%%%%%%%%%%%%%%%%%%%%%%%%%%%%%%%%%%%%%%%%%%%%%%%%
\subsection{} 
The proof of Theorem~\ref{thm:cover} proceeds by 
transitivity of covers~\cite[Proposition~8.5]{BK1}. Putting 
$\l_\L^\so=\l^\so_\P|_{\J_\P^\so\cap\L}$, the first step is to show:

\begin{lemm}\label{lem:LneG}
The pair $(\J_\P^\so,\l^\so_\P)$ is a cover of $(\J_\P^\so\cap\L,\l^\so_\L)$. 
Moreover, there is a support-preserving Hecke algebra isomorphism
$$
\Hh(\G,\l^\so_\P) \simeq \Hh(\L,\l^\so_\L).
$$
\end{lemm}

\begin{proof} 
Since $(\J_\P^\so,\l^\so_\P)$ is a decomposed pair above 
$(\J_\P^\so\cap\M,\l^\so_\M)$ and $\M\subseteq\L\subseteq\G$, it is 
certainly also a decomposed pair above $(\J_\P^\so\cap\L,\l^\so_\L)$.

Now the support of the Hecke algebra~$\Hh(\G,\l^\so_\P)$ is the 
intertwining of~$\l^\so_\L$, which is contained in the intertwining 
of~$\th_\P$. By Lemma~\ref{lem:thetaP}(i), this intertwining 
is~$\J_\P\G_\E\J_\P \subseteq \J^\so_\P\L\J^\so_\P$. The result now 
follows from~\cite[Theorem~7.2]{BK1}.
\end{proof}

%%%%%%%%%%%%%%%%%%%%%%%%%%%%%%%%%%%%%%%%%%%%%%%%%%%%%%%%%%%%%%%%%%%%%%%%
\subsection{} Lemma~\ref{lem:LneG} reduces us to proving 
that~$(\J_\P^\so\cap\L,\l_\L^\so)$ is a cover 
of~$(\J_\P^\so\cap\M,\l_\M^\so)$.
By Proposition~\ref{prop:kappaL}, we 
have~$\l_\L^\so=\k'_{\P\cap\L}\otimes\rho_\M^\so$, 
where~$\k'=\k_\Q|_{\J^\so\cap\L}$ is a standard~$\b$-extension 
of~$\th|_{\H^1\cap\L}$, and we think of~$\rho_\M^\so$ as a representation 
of~$\left(\J_\P^\so\cap\L\right)/\left(\J^1_\P\cap\L\right)\simeq 
\left(\J^\so\cap\M\right)/\left(\J^1\cap\M\right)$. The first step is 
to describe~$\k'_{\P\cap\L}$, whence~$\l^\so_\L$, more carefully.

Recall that~$V=\bigoplus_{i=-l}^l\V^i$ is the self-dual decomposition 
associated to the semisimple stratum~$[\La,n,0,\b]$, and that, for 
each~$i$, we have~$\V^i=\bigoplus_{j\in J_i}\W_j$, for some subset~$J_i$ 
of~$\{-m,\ldots,m\}$. Writing~$\ee^i$ for the projection onto~$\V^i$ as 
usual, and~$\La^i=\La\cap\V^i$, the stratum~$[\La^i,n_i,0,\ee^i\b\ee^i]$ 
is 
\begin{itemize}
\item[$\bullet$] skew semisimple, for~$i=0$,
\item[$\bullet$] simple, for~$i\ne 0$,
\end{itemize}
where~$n_i=-\v_\La(\ee^i\b\ee^i)=-\v_{\La^i}(\ee^i\b\ee^i)$.

We write~$\L=\G^0\times\prod_{i=1}^l\tG^i$, where~$\tG^i=\Aut_\F(\V^i)$. 
We have~$\th|_{\H^1\cap\L}=\th'_0\otimes\bigotimes_{i=1}^l\tth'_i$, 
where~$\th'_0$ is a skew semisimple character 
in~$\Cc_-(\La^0,0,\ee^0\b\ee^0)$, and~$\tth'_i$ is a simple character 
in~$\Cc(\La^i,0,2\ee^i\b\ee^i)$. Then the 
standard~$\b$-extension~$\k'$ takes the 
form~$\k'=\k'_0\otimes\bigotimes_{i=1}^l\til\k'_i$, for~$\k'_0$ a 
standard~$\ee^0\b\ee^0$-extension of~$\th'_0$, and~$\til\k'_i$ 
a~$2\ee^i\b\ee^i$-extension of~$\tth'_i$.

Since~$\M\subseteq\L$, we have~$\P\cap\L=\P^0\times\prod_{i=1}^l\til\P^i$, 
with~$\P^i$ a parabolic subgroup of~$\G^i$. We 
put~$\rho'_0=\rho^\so_0\otimes\bigotimes_{j\in J_0,\,j>0}\til\rho_j$, 
and~$\til\rho'_i=\bigotimes_{j\in J_i}\til\rho_j$, for~$i>0$. Then we 
put~$\l^\so_0=\k'_0\otimes\rho'_0$, 
and~$\til\l'_i=\til\k'_i\otimes\til\rho'_i$, for~$i>0$.

Now~$\J_\P^\so\cap\L=\J^\so_{\P^0}\times\prod_{i=1}^l\til\J_{\til\P^i}$ 
(with the obvious notation) and~$\k'_{\P\cap\L}\simeq
\k'_{0,\P^0}\otimes\bigotimes_{i=1}^l\til\k'_{i,\til\P^i}$. In particular, 
we also get~$\l^\so_\L\simeq\l^\so_{0,\P^0}\otimes
\bigotimes_{i=1}^l\til\l'_{i,\til\P^i}$.

Finally, we write~$\M^0=\M\cap\G^0$ and~$\til\M^i=\M\cap\tG^i$, for~$i>0$, 
so that~$\M=\M^0\times\prod_{i=1}^l\til\M^i$. Then, in order to prove 
that~$(\J_\P^\so\cap\L,\l_\L^\so)$ is a cover 
of~$(\J^\so\cap\M,\l_\M^\so)$ we need to show: 
\begin{itemize}
\item[$\bullet$] $(\J^\so_{\P^0},\l^\so_{0,\P^0})$ is a cover 
of~$(\J^\so_{\P^0}\cap\M^0,\l^\so_{0,\P^0}|_{\J_{\P^0}\cap\M^0})$; and
\item[$\bullet$] $(\til\J_{\til\P^i},\til\l'_{i,\til\P^i})$ is a cover 
of~$(\til\J_{\til\P^i}\cap\til\M^i, \til\l'_{i,\til\P^i}
|_{\til\J_{\til\P^i}\cap\til\M^i})$, for~$i>0$.
\end{itemize}
The latter is given by~\cite[Proposition~8.1]{SS6}: since the 
underlying stratum is simple, it is a~\emph{homogeneous semisimple type}, 
in the sense of~\cite{BKsemi,SS6}. On the other hand, the former is the case 
of a~\emph{skew} semisimple stratum; that is, we have reduced the proof of
Theorem~\ref{thm:cover} to the case~$\L=\G$, and we are in the situation 
of~\cite[\S7]{S5}. Indeed it is possible to extract the proof that we get 
a cover here from the results in \emph{loc.\ cit.}, which we will do 
in the following subsections.

%%%%%%%%%%%%%%%%%%%%%%%%%%%%%%%%%%%%%%%%%%%%%%%%%%%%%%%%%%%%%%%%%%%%%%%%
\subsection{}\label{SS.LeqGbutsplit}
We are now in the situation of Theorem~\ref{thm:cover} in the special 
case~$\L=\G$, so that~$[\La,n,0,\b]$ is a skew semisimple stratum. 
In~\cite[\S6.3]{S5}, an involution~$\s_j$ is defined on~$\tG_j$, 
for~$j>0$, coming from the composition of the involution~$\s$ on~$\G$ 
and a Weyl group element which exchanges~$\W_j$ with~$\W_{-j}$. 
By~\cite[Lemma~6.9, Corollary~6.10]{S5}, the 
group~$\til\J(\b_j,\La_j)$ is stable under 
this involution, and~$\til\k_j\simeq\til\k_j\circ\s_j$. 

Recall that we 
have~$\rho^\so_\M=\rho^\so_0\otimes\bigotimes_{j=1}^m\til\rho_j$. 
For~$j>0$ we put~$\til\rho_{-j}=\til\rho_j\circ\s_j$.

We suppose first that there is an index~$k>0$ such 
that~$\til\rho_k\not\simeq\til\rho_{-k}$. We put 
\[
J_1=\{-m\le j\le m \mid \til\rho_j\simeq\til\rho_k\},\quad 
J_0=\{j\mid\pm j\not\in J_1\},\quad J_{-1}=\{-j\mid j\in J_1\},
\]
and set~$\Y^i=\bigoplus_{j\in J_i}\W_j$, for~$i=-1,0,1$. 
Since~$\til\rho_k\not\simeq\til\rho_{-k}$ we 
have~$\V=\Y^{-1}\oplus\Y^0\oplus\Y^1$. Let~$\M'$ be the Levi subgroup 
of~$\G$ stabilizing this decomposition and let~$\P'=\M'\U'$ be a parabolic 
subgroup containing~$\P$. (Note that one may need to change the choice of 
the parabolic subgroup~$\P$ in order to achieve this.) We 
have~$\M'=\G^0\times\tG^1$, where~$\G^0=\Aut_\F\Y^0\cap\G$ 
and~$\tG^1=\Aut_\F\Y^1$, and write~$\M=\M^0\times\til\M^1$ also.

By~\cite[Proposition~7.10]{S5} and its proof we have:

\begin{lemm}\label{lem:LeqGbutsplit}
The pair $(\J_\P^\so,\l_\P^\so)$ is a 
cover of~$(\J_\P^\so\cap\M',\l_\P^\so|_{\J_\P^\so\cap\M'})$ and 
there is a support-preserving isomorphism of Hecke 
algebras~$\Hh(\G,\l_\P^\so)\simeq\Hh(\M',\l_\P^\so|_{\J_\P^\so\cap\M'})$.
\end{lemm}

Now we have~$\J_\P^\so\cap\M'=(\J_\P^\so\cap\G^0)\times(\J_\P^\so\cap\tG^1)$ 
and, as in the previous section, we need to prove:
\begin{itemize}
\item[$\bullet$] $(\J^\so_{\P}\cap\G^0,\l^\so_{\P}|_{\J^\so_{\P}\cap\G^0})$ 
is a cover of~$(\J^\so_{\P}\cap\M^0,\l^\so_{\P}|_{\J^\so_{\P}\cap\M^0})$; and
\item[$\bullet$] $(\J^\so_{\P}\cap\tG^1,\l^\so_{\P}|_{\J^\so_{\P}\cap\tG^1})$ 
is a cover of~$(\J^\so_{\P}\cap\til\M^1,\l^\so_{\P}|_{\J^\so_{\P}\cap\til\M^1})$.
\end{itemize}
Again as in the previous section, the latter is a cover 
by~\cite[Proposition~6.7]{SS6}; it is a simple type. The former is again 
the case of a skew semisimple stratum, but with fewer indices~$j$ such 
that~$\til\rho_j\not\simeq\til\rho_{-j}$. In particular, by repeating 
the process in this paragraph, we can reduce to the case 
where~$\til\rho_j\simeq\til\rho_{-j}$, for all~$j$.

%%%%%%%%%%%%%%%%%%%%%%%%%%%%%%%%%%%%%%%%%%%%%%%%%%%%%%%%%%%%%%%%%%%%%%%%
\subsection{}\label{SS.SO11}
We suppose now that~$\G_\E$ does not have compact centre. This implies 
that~$\G$ is a special orthogonal group, that~$\b_k=0$ for a unique~$k>0$, 
and that~$\dim_\F \W_k=1$. In this case set~$\Y^1=\W_k$, $\Y^{-1}=\W_{-k}$, 
and~$\Y^0=\bigoplus_{j\ne\pm k}\W_j$, let~$\M'$ be the Levi subgroup 
stabilizing the decomposition~$\V=\Y^{-1}\oplus\Y^0\oplus\Y^{1}$, and 
let~$\P'=\M'\U'$ be a parabolic subgroup containing~$\P$. (Again, this may require the choice of~$\P$ to be changed.) We have~$\G_\E\subseteq\M'$ and, 
by~\cite[Corollary~6.16]{S5}, 
$\I_\G(\l^\so_\P)\subseteq\J^\so_\P\M'\J^\so_\P$. In particular we get:

\begin{lemm}\label{lem:LeqGSO11}
The pair $(\J_\P^\so,\l_\P^\so)$ is a 
cover of~$(\J_\P^\so\cap\M',\l_\P^\so|_{\J_\P^\so\cap\M'})$ and 
there is a support-preserving isomorphism of Hecke 
algebras~$\Hh(\G,\l_\P^\so)\simeq\Hh(\M',\l_\P^\so|_{\J_\P^\so\cap\M'})$.
\end{lemm}

As in previous sections, this reduces us to the case where~$\G_\E$ has 
compact centre.

%%%%%%%%%%%%%%%%%%%%%%%%%%%%%%%%%%%%%%%%%%%%%%%%%%%%%%%%%%%%%%%%%%%%%%%%
\subsection{}\label{SS.nonsplit}
We have finally reduced to the case 
where~$\til\rho_j\simeq\til\rho_{-j}$, for all~$j$ and~$\G_\E$ has 
compact centre; this is exactly the situation
of~\cite[\S7.2.2]{S5}. Moreover, by changing~$\P$ if necessary, we may
assume the parabolic subgroup is the same one as in~\emph{loc.\ cit.}
In~\cite[\S7.2.2]{S5}, two auxiliary~$\oe$-lattice sequences~$\MM_t$,
$t=0,1$, are defined, along with Weyl group
elements~$s_t\in\P(\MM_{t,\oe})$, which we describe below, along with
some auxiliary elements. We
have~$\G_\E=\prod_{i\in\I_0}\G_{\E_i}$ and we will
write~$\I_0=\{1,\ldots,l\}$, to match the notation
of~\cite[\S7.2.2]{S5}; then~$\W^{(m)}\subset\V^\lprime$, with~$1\le \lprime\le l$ 
maximal such that~$\V^\lprime$ contains some~$\W^{(j)}$, and~$\b_i\ne 0$ for~$i>1$.

We put~$\W^{(\lprime,0)}=\V^\lprime\cap\W_0$ and denote
by~$\La^{(\lprime,0)}$ the~$\fo_{\E_\lprime}$-lattice sequence~$\La\cap\W^{(\lprime,0)}$.
Let~$p_\La\in\P(\La^{(\lprime,0)}_{\fo_{\E_\lprime}})$ be an element of order at
most~$2$ such that the
quotient~$\P(\La^{(\lprime,0)}_{\fo_{\E_\lprime}})/\P^\so(\La^{(\lprime,0)}_{\fo_{\E_\lprime}})$
(which has order~$1$ or~$2$) is generated by the image
of~$p_\La$. Then also~$\P(\La^{\lprime}_{\fo_{\E_\lprime}})/\P^\so(\La^{\lprime}_{\fo_{\E_\lprime}})$
is generated by the image of~$p_\La$. We split into cases.
\begin{enumerate}
\item Suppose either that~$\G_{\E_\lprime}$ is~\emph{not} an orthogonal
group, or that~$\dim_{\E_\lprime}\W_m$ is even. Then~$s_0,s_1$ are the
elements denoted~$s_m,s_m^\varpi$ respectively
in~\emph{loc.\ cit.}. Note that~$p_\La$ commutes with both~$s_0$
and~$s_1$.

In this situation, it is straightforward to check, using the
definitions of the elements in~\cite[\S6.2]{S5},
that~$s_t\in\P^\so(\MM_{t,\oe})$ unless~$\E_\lprime/\E_{\lprime,\so}$ is
ramified,~$m$ is odd (in which case~$m=1$) and~$\e= (-1)^t$. Moreover,
if~$s_t\not\in\P^\so(\MM_{t,\oe})$ then either~$p_\La
s_t\in\P^\so(\MM_{t,\oe})$ or else~$\P^\so(\MM_{t,\oe})=\P^\so(\La_{\oe})$,
in which case~$\P(\MM^\lprime_{t,\oe})/\P_1(\MM^\lprime_{t,\oe})$ has the
form~$\O(1,1)(k_{\E_\lprime})\times\Gg$, for~$\Gg$ some product of
connected finite reductive groups,
while~$\P(\La^\lprime_{\oe})/\P_1(\La^\lprime_{\oe})$ has the
form~$\SO(1,1)(k_{\E_\lprime})\times\Gg$.
\item If~$\G_{\E_\lprime}$ is (special) orthogonal (so that~$\E_\lprime=\F$
and~$\e=+1$) and~$\dim_{\F}\W_m$ is odd, the choice of~$\P$ and the
property that~$\til\rho_j\simeq\til\rho_{-j}$, for all~$j$, mean
that~$\lprime=1$ and~$\dim_{\F}\W_k=1$, for all~$k>0$, and there are two cases.
\begin{enumerate}
\item If~$\G^+_{\E_1}\cap\Aut_{\F}(\W_0)\ne 1$, then we
have~$\P^\so(\La_{\oe})/\P_1(\La_{\oe})\simeq\GL_1(k_{\F})\times\Gg_0^\so\times\Gg_1^\so\times\Gg^\so$,
where each~$\Gg_t^\so$ is a special orthogonal group over~$k_{\F}$
(one of which may be trivial) and~$\Gg^\so$ is some product of
connected finite reductive groups. If~$\Gg_t^\so$ is non-trivial, there is an
element~$p_t\in\P^+(\La^{(1,0)})\setminus\P(\La^{(1,0)})$ such
that~$p_t^2=1$, which commutes with both~$s_m,s_m^\varpi$, and whose image
in~$\P^+(\La_{\oe})/\P_1(\La_{\oe})\simeq\GL_1(k_{\F})\times\Gg_0\times\Gg_1\times\Gg$
lies in the orthogonal group~$\Gg_t$ (whose connected component
is~$\Gg_t^\so$). If~$\Gg_t^\so$ is trivial, we put~$p_t=p_{1-t}$; in
any case,~$p_0,p_1$ commute. Moreover, we can assume that~$p_\La=p_0p_1$.

If exactly one of~$p_0,p_1$ normalizes the representation~$\rho_\M^\so$,
viewed as a representation of~$\P^\so(\La_{\oe})$ trivial
on~$\P_1(\La_{\oe})$, then we set~$p$ to be this element; if both or neither
normalize, then we arbitrarily choose~$p$ to be one of
them. Then~$s_0,s_1$ are the elements~$ps_m,ps_m^\varpi$ respectively,
which lie in~$\G_\E$. 

Note that~$s_t\in\P^\so(\MM_{t,\oe})$ precisely when~$\Gg_t^\so$ is
non-trivial and~$p=p_t$. If~$\Gg_t^\so$ is trivial, then~$\P^\so(\MM_{t,\oe})/\P_1(\MM_{t,\oe})\simeq\SO(1,1)(k_{\F})\times\Gg_{1-t}^\so\times\Gg^\so$.
\item Otherwise,~$s_0,s_1$ are the elements
denoted~$s_ms_{m-1},s_m^\varpi s_{m-1}^\varpi$ respectively
in~\cite[\S7.2.2]{S5}. Note that in this
case~$m\ge 2$, since~$\G_{\E_1}$ has compact centre so we cannot
have~$\G_{\E_1}\simeq\SO(1,1)(\F)$.
\end{enumerate}
\end{enumerate}

In all cases but case (ii)(b), we set~$\Y^1=\W_m$,~$\Y^{-1}=\W_{-m}$ and $\Y^0=\sum_{j\ne\pm m}\W_j$; in the exceptional case we set~$\Y^1=\W_m\oplus\W_{m-1}$,~$\Y^{-1}=\W_{-m}\oplus\W_{1-m}$ and~$\Y^0=\sum_{j\ne\pm m,m-1}\W_j$. Denote by~$\M'$ the Levi subgroup stabilizing the decomposition~$\V=\Y^{-1}\oplus\Y^0\oplus\Y^{1}$, and let~$\P'=\M'\U'$ be a parabolic subgroup containing~$\P$. We deal with an easy case first.

\begin{lemm}\label{lem:iiia}
Suppose we are in case (ii)(a) and neither~$p_0$ nor~$p_1$
normalizes~$\rho_\M^\so$. Then the pair
$(\J_\P^\so,\l_\P^\so)$ is a cover
of~$(\J_\P^\so\cap\M',\l_\P^\so|_{\J_\P^\so\cap\M'})$ and there is a
support-preserving isomorphism of Hecke
algebras~$\Hh(\G,\l_\P^\so)\simeq\Hh(\M',\l_\P^\so|_{\J_\P^\so\cap\M'})$.
\end{lemm}

\begin{proof} 
By~\cite[Corollary~6.16]{S5}, we
have~$\I_\G(\l^\so_\P)\subseteq\J^\so_\P\M'\J^\so_\P$, and the result
follows as usual, as in Lemma~\ref{lem:LneG}.
\end{proof}

Now suppose we are not in the case of Lemma~\ref{lem:iiia}.
For~$t=0,1$, denote by~$\k_t$ a $\b$-extension of~$\eta$ compatible with some
standard~$\b$-extension of~$\J^\so(\b,\MM_t)$. By~\cite[Corollary~6.13]{S5}, we have
\[
\k_t\simeq \Ind_{\J^\so_\P}^{\J^\so} \k_\P\otimes\chi_t,
\]
for some self-dual character~$\chi_t$. We write~$\rho^\so_t=\rho^\so_M\otimes\chi_t^{-1}$, which is still a self-dual cuspidal representation. Moreover, by~\cite[(7.3)]{S5}, there is a support-preserving injective algebra map
\begin{equation}\label{eqn:Heckeinj}
\Hh(\P(\MM_{t,\oe}),\rho_t^\so)\hookrightarrow \Hh(\G,\l_\P^\so),
\end{equation}
where~$\rho_t^\so$ being regarded as a cuspidal representation of~$\P(\La_{\oe})$. By~\cite[Theorem~7.12]{M1}, there is an invertible element in~$\Hh(\P(\MM_{t,\oe}),\rho_t^\so)$ with support~$\P(\La_{\oe})s_t\P(\La_{\oe})$, and we denote by~$T_t$ its image in~$\Hh(\G,\l_\P^\so)$. 

By~\cite[Lemmas~7.11,7.12]{S5}, for a suitable integer~$e$, the element~$(T_0T_1)^e$ is an invertible element of~$\Hh(\G,\l_\P^\so)$ supported on the double coset of a strongly $(\P,\J_\P^\so)$-positive element of the centre of~$\M'$. Indeed, we have:

\begin{prop}[{\cite[Proposition~7.13]{S5}}]
The pair~$(\J_\P^\so,\l_\P^\so)$ is a 
cover of~$(\J_\P^\so\cap\M',\l_\P^\so|_{\J_\P^\so\cap\M'})$.
\end{prop}

As in previous sections, we have~$\M'=\G^0\times\tG^1$, where~$\G^0=\Aut_\F\Y^0\cap\G$ and~$\tG^1=\Aut_\F\Y^1$, and we write~$\M=\M^0\times\til\M^1$.
Then~$\J_\P^\so\cap\M'=(\J_\P^\so\cap\G^0)\times(\J_\P^\so\cap\tG^1)$ 
and we need to prove:
\begin{itemize}
\item[$\bullet$] $(\J^\so_{\P}\cap\G^0,\l^\so_{\P}|_{\J^\so_{\P}\cap\G^0})$ 
is a cover of~$(\J^\so_{\P}\cap\M^0,\l^\so_{\P}|_{\J^\so_{\P}\cap\M^0})$; and
\item[$\bullet$] $(\J^\so_{\P}\cap\tG^1,\l^\so_{\P}|_{\J^\so_{\P}\cap\tG^1})$ 
is a cover of~$(\J^\so_{\P}\cap\til\M^1,\l^\so_{\P}|_{\J^\so_{\P}\cap\til\M^1})$.
\end{itemize}
Again as previously, the latter is a cover 
by~\cite[Proposition~6.7]{SS6}; it is a simple type. (In fact, except in case~(ii)(b) above, we have~$\til\M^1=\tG^1$.)

The former is again the case of a skew semisimple stratum, but with smaller~$m$. In particular, by repeating the process in this paragraph, we reduce to the case~$m=0$, in which case~$\M^0=\G^0$ and there is nothing left to do.

%%%%%%%%%%%%%%%%%%%%%%%%%%%%%%%%%%%%%%%%%%%%%%%%%%%%%%%%%%%%%%%%%%%%%%%%
%%%%%%%%%%%%%%%%%%%%%%%%%%%%%%%%%%%%%%%%%%%%%%%%%%%%%%%%%%%%%%%%%%%%%%%%
%%%%%%%%%%%%%%%%%%%%%%%%%%%%%%%%%%%%%%%%%%%%%%%%%%%%%%%%%%%%%%%%%%%%%%%%
%%%%%%%%%%%%%%%%%%%%%%%%%%%%%%%%%%%%%%%%%%%%%%%%%%%%%%%%%%%%%%%%%%%%%%%%

\section{Hecke algebras}\label{S.hecke}

In this section we prove Theorem~\ref{thm:maximal} of the introduction:
that is, we describe the Hecke algebra of a cover in the case that
$\tau$ is a cuspidal irreducible representation of a \emph{maximal}
proper Levi subgroup $\M$ of $\G$ (so~$\M$ is the stabilizer of a 
self-dual decomposition~$\V=\W_{-1}\oplus\W_0\oplus\W_1$), up to the 
computation of some parameters. We will also explain how, in principle, 
these parameters can be computed. 

As in the introduction, we 
write~$\ss=[\M,\tau]_\G$ and~$\ss_\M=[\M,\tau]_\M$ and put
\[
\N_\G(\ss_\M) = \{g\in\N_\G(\M) : 
{}^g\tau\hbox{ is inertially equivalent to }\tau\}.
\]
We also put~$\WW_\ss=\N_\G(\ss_\M)/\M$, a subgroup of the  
group~$\N_\G(\M)/\M$ of order~$2$.

We denote by~$(\J_\P,\l_\P)$ the~$\ss$-type constructed in the previous 
section, and put~$\J_\M=\J_\P\cap\M$ and~$\l_\M=\l_\P|_{\J_\P\cap\M}$, so 
that~$(\J_\P,\l_\P)$ is a cover of the~$\ss_\M$-type~$(\J_\M,\l_\M)$.

%%%%%%%%%%%%%%%%%%%%%%%%%%%%%%%%%%%%%%%%%%%%%%%%%%%%%%%%%%%%%%%%%%%%%%%%
\subsection{} Suppose first that~$\N_\G(\ss_\M)=\M$, so 
that~$\WW_\ss$ is trivial. In this case, by~\cite[Theorem~1.5]{BKcover}, 
we have an isomorphism
$$
\Hh(\M,\l_\M) \to \Hh(\G,\l_\P).
$$
Moreover~$\Hh(\M,\l_\M)$ is abelian, isomorphic to~$\CC[X^{\pm 1}]$ and 
the result follows.

%%%%%%%%%%%%%%%%%%%%%%%%%%%%%%%%%%%%%%%%%%%%%%%%%%%%%%%%%%%%%%%%%%%%%%%%
\subsection{}\label{SS.nonsplithecke} 
Now suppose that~$\N_\G(\ss_\M)\ne\M$, so 
that~$\WW_\ss$ has order~$2$. We note first that, in this situation, 
we cannot have a support-preserving 
isomorphism~$\Hh(\M,\l_\M) \to \Hh(\G,\l_\P)$ since 
the induced representation~$\Ind_\P^\G\tau\otimes\chi$ reduces for 
some unramified character~$\chi$ of~$\M$. This implies that we also do not have a support-preserving isomorphism~$\Hh(\M,\l_\M^\so) \to \Hh(\G,\l_\P^\so)$.

We now proceed through the construction of~\S\ref{S.covers} and we use 
all the notation from there. Note that we must have~$\L=\G$, or else we 
would have~$\L=\M$ and Lemma~\ref{lem:LneG} would give us an 
isomorphism~$\Hh(\M,\l_\M^\so) \to \Hh(\G,\l_\P^\so)$. Similarly, we cannot 
be in the situation of~\S\ref{SS.LeqGbutsplit} or~\S\ref{SS.SO11}, by 
Lemmas~\ref{lem:LeqGbutsplit},~\ref{lem:LeqGSO11}. 

Thus we are in the situation of~\S\ref{SS.nonsplit}, whose notation we adopt. Further, we are not in the exceptional case~(ii)(b), since~$\M$ is a maximal Levi subgroup, nor in the case of Lemma~\ref{lem:iiia} since we do not have an isomorphism~$\Hh(\M,\l_\M^\so) \to \Hh(\G,\l_\P^\so)$. The lattice sequence~$\MM_1$ is just the standard lattice sequence~$\MM_{\La}$ used to define the standard~$\b$-extension~$\k$. In particular, the character~$\chi_1$ is trivial. Moreover, as in~\cite[\S2.3]{GKS}, by changing~$\k_0$ if necessary, we may assume that~$\chi_0$ is a quadratic character. 

Recall the element~$p_\La\in\P(\La^{(l,0)}_{\fo_{\E_\lprime}})$ defined
in~\S\ref{SS.nonsplit}: its image generates the
quotient~$\P(\La^{l}_{\fo_{\E_\lprime}})/\P^\so(\La^{l}_{\fo_{\E_\lprime}})$. We
define~$\J_\P^*=\P(\La^{l}_{\fo_{\E_\lprime}})\J_\P^\so$, which
contains~$\J_\P^\so$ with index at most~$2$. We fix~$t\in\{0,1\}$ and split according to 
the cases of~\S\ref{SS.nonsplit}, which we further subdivide.

\begin{enumerate}
\item Suppose either that~$\G_{\E_\lprime}$ is~\emph{not} an orthogonal
group, or that~$\dim_{\E_\lprime}\W_1$ is even.
\end{enumerate}

(a) Assume first that~$s_t\in\P^\so(\MM_{t,\oe})$. We denote by~$\Gg_t$ the connected finite reductive group~$\P^\so(\MM_{t,\oe})/\P_1(\MM_{t,\oe})$, and regard the representation~$\rho_{\M}^\so\chi_t$ as the inflation to the parabolic subgroup~$\Pp_t=\P^\so(\La_{\oe})/\P_1(\MM_{t,\oe})$ of a cuspidal representation of the Levi subgroup~$\P^\so(\La_{\oe})/\P_1(\La_{\oe})$. From~\eqref{eqn:Heckeinj}, we get an injection of Hecke algebras
\[
\Hh(\Gg_t,\rho_{\M}^\so\chi_t)\hookrightarrow \Hh(\G,\l_\P^\so).
\]
The element denoted~$T_t$ in~\S\ref{SS.nonsplit} is the image of an invertible element~$\bar T_t\in\Hh(\Gg_t,\rho_{\M}^\so\chi_t)$ which satisfies a quadratic relation. This quadratic relation is given explicitly (in principle) by~\cite[Theorem~8.6]{L}. By scaling~$T_t$ if necessary, we may assume that the relation takes the form
$$
(T_t-q_t)(T_t+1)=0,
$$
and then, by~\cite[Theorem~8.6]{L}, $q_t$ is a power of~$q_\so$; indeed, by~\cite[Theorem~4.14]{HL}, it can also be described as the quotient of the dimensions of the two irreducible components of~$\Ind_{\Pp_t}^{\Gg_t}\rho_{\M}^\so\chi_t$. 

Now we induce to~$\J_\P^*$. Let~$\l_\P^*$ be an irreducible component
of~$\Ind_{\J_\P^\so}^{\J_\P^*}\l_\P^\so$ contained
in~$\l_\P$. If~$\l_\P^*|_{\J_{\P}^\so}$ is reducible (equivalently,
if~$p_\La$ does not normalize~$\l_\M^\so$)
then~$\l_\P^*\simeq\Ind_{\J_\P^\so}^{\J_\P^*}\l_\P^\so$. Then,
by~\cite[(4.1.3)]{BK}, we have a support-preserving isomorphism 
\[
\Hh(\G,\l_\P^\so)\simeq \Hh(\G,\l_\P^*),
\]
and we denote by~$T_t^*$ the image of~$T_t$ under this isomorphism,
which satisfies the same quadratic relation.

Otherwise,~$\l_\P^*|_{\J_{\P}^\so}$ is irreducible,~$p_\La$
normalizes~$\l_\M^\so$, and~$\Ind_{\J_\P^\so}^{\J_\P^*}\l_\P^\so$ has
two inequivalent irreducible components~$\l_\P^*$ and~$\l_\P'$. We can
identify~$\Hh(G,\l_\P^*)$ and~$\Hh(G,\l'_\P)$ as subalgebras
of~$\Hh(\G,\Ind_{\J_\P^\so}^{\J_\P^*}\l_\P^\so)$, canonically
since~$\Ind_{\J_\P^\so}^{\J_\P^*}\l_\P^\so$ is multiplicity free. 
Note also that~$(\J_\P^*,\l_\P^*)$ is a cover
of~$(\J_\P^*\cap\M,\l_\P^*|_{\J_\P^*\cap\M})$, by~\cite[Lemma~3.9]{M3}, 
and the same applies to~$\l'_\P$. 
Finally, since~$s_t$ normalizes both restrictions~$\l_\P^*|_{\J_\P^*\cap\M}$ 
and~$\l_\P'|_{\J_\P^*\cap\M}$, the image of~$T_t$ under the support-preserving 
isomorphism
\[
\Hh(\G,\l_\P^\so)\simeq \Hh(\G,\Ind_{\J_\P^\so}^{\J_\P^*}\l_\P^\so)
\]
decomposes as~$T_t^*+T'_t$, with~$T_t^*\in\Hh(G,\l_\P^*)$
and~$T'_t\in\Hh(G,\l'_t)$ each satisfying the same relation as~$T_t$.

In either case, when~$s_t\in\P^\so(\MM_{t,\oe})$, we end with an invertible
element~$T_t^*\in\Hh(G,\l_\P^*)$ supported on~$\J_\P^*s_t\J_\P^*$ and satisfying 
a quadratic relation of the required form, with computable parameter~$q_t$.

(b) Now suppose that~$s_t\not\in\P^\so(\MM_{t,\oe})$. If~$p_\La$
normalizes~$\l_\M^\so$ and~$p_\La s_t\in\P^\so(\MM_{t,\oe})$, we can
replace~$s_t$ by~$p_\La s_t$ and argue exactly as in the previous
case to get an element~$T_t^*\in\Hh(G,\l_\P^*)$ as required.

(c) Suppose now that~$p_\La s_t\in\P^\so(\MM_{t,\oe})$ but~$p_\La$ does
not normalize~$\l_\M^\so$. We write~$\P^*(\La_\oe)$ for the group
generated by~$p_\La$ and~$\P^\so(\La_\oe)$, so
that~$\J_\P^*=\P^*(\La_\oe)\J^1_\P$. Then, the quotient
group~$\P^*(\La_\oe)/\P_1(\La_\oe)$ has the
form~$\GL_1(k_{\E_\lprime})\times\Gg_t\times\Gg^\so$, for~$\Gg_t$ some 
orthogonal group over~$k_{\E_\lprime}$ and~$\Gg^\so$ a product of connected
finite reductive groups, while~$\P^\so(\La_\oe)/\P_1(\La_\oe)\simeq
\GL_1(k_{\E_\lprime})\times\Gg^\so_t\times\Gg^\so$, where~$\Gg^\so_t$ is the
connected component of~$\Gg_t$. 

We also set~$\P^*(\MM_{t,\oe})=\P^*(\La_\oe)\P^\so(\MM_{t,\oe})$. Then we 
have~$\P^*(\MM_{t,\oe})/\P_1(\MM_{t,\oe})\simeq
\Gg_{1,t}\times\Gg^\so$, where~$\Gg_{1,t}$ is an orthogonal group
over~$k_{\E_\lprime}$ with Levi subgroup~$\GL_1(k_{\E_\lprime})\times\Gg_t$,
and~$\P^\so(\MM_{t,\oe})/\P_1(\MM_{t,\oe})\simeq
\Gg^\so_{1,t}\times\Gg^\so$, where~$\Gg^\so_{1,t}$, the connected
component of~$\Gg_{1,t}$, is a special orthogonal group
over~$k_{\E_\lprime}$ with Levi subgroup~$\GL_1(k_{\E_\lprime})\times\Gg^\so_t$.

We write the image of~$\P^*(\La_\oe)$ in~$\P^*(\MM_{t,\oe})/\P_1(\MM_{t,\oe})$ 
as~$\Pp_t\times\Gg^\so$, where~$\Pp_t$ is a parabolic subgroup of~$\Gg_{1,t}$ 
with Levi component~$\GL_1(k_{\E_\lprime})\times\Gg_t$. Similarly, we write the image 
of~$\P^\so(\La_\oe)$ as~$\Pp^\so_t\times\Gg^\so$. We have the following picture:
\[
\xymatrix{\Pp^\so_t\times\Gg^\so \ar[r]^{\Ind}\ar[d]_{\Ind} & 
\Gg^\so_{1,t}\times\Gg^\so\ar[d]^{\Ind} \\
\Pp_t\times\Gg^\so \ar[r]_{\Ind} & \Gg_{1,t}\times\Gg^\so
}
\]
Since (the image of)~$p_\La$ does not normalize~$\l_\M^\so$, but does
normalize the~$\b$-extension~$\k_t$, it also does not
normalize~$\rho_\M^\so\chi_t$ and
hence~$\rho_\M^*:=\Ind_{\Pp^\so_t\times\Gg^\so}^{\Pp_t\times\Gg^\so}\rho_\M^\so\chi_t$ 
is irreducible. Similarly, since~$s_t\not\in\P^\so(\MM_{t,\oe})$, the
induced
representation~$\Ind_{\Pp^\so_t\times\Gg^\so}^{\Gg^\so_{1,t}\times\Gg^\so}\rho_\M^\so\chi_t$
is also irreducible. On the other hand, since~$s_t$
intertwines~$\rho_\M^\so\chi_t$, the induced
representation~$\Ind_{\Pp_t\times\Gg^\so}^{\Gg_{1,t}\times\Gg^\so}\rho_\M^*$
is reducible. By restricting back to~$\Gg^\so_{1,t}\times\Gg^\so$, we see
that it must reduce as a direct sum of two inequivalent irreducible
representations of the same dimension. Thus there is an element~$\bar
T_t^*\in\Hh(\Gg_{1,t}\times\Gg^\so,\rho_\M^*)$ satisfying~$(\bar
T_t^*)^2=1$. Finally, by~\cite[(7.3)]{S5}, there is again a 
support-preserving injective algebra map
\[
\Hh(\P(\MM_{t,\oe}),\rho_\M^*)\hookrightarrow \Hh(\G,\l_\P^*),
\]
and we find an invertible element~$T_t^*\in\Hh(G,\l_\P^*)$ satisfying a 
quadratic relation 
$$
(T^*_t-1)(T^*_t+1)=0,
$$
and~$q_t=q_{\so}^0=1$.

(d) Finally, suppose that~$s_t\not\in\P^\so(\MM_{t,\oe})$ but~$p_\La=1$,
so that~$\J_\P^*=\J_\P^\so$. The argument here is very similar. In this case we
have~$\P(\MM_{t,\oe})/\P_1(\MM_{t,\oe})\simeq\O(1,1)(k_{\E_\lprime})\times\Gg$,
for some product of (possibly non-connected) finite reductive groups,
and the image of~$s_t$ lies in~$\O(1,1)(k_{\E_\lprime})$. We denote
by~$\Gg_t$ the non-connected group~$\O(1,1)(k_{\E_\lprime})\times\Gg^\so$,
where~$\Gg^\so$ is the connected component of~$\Gg$. The image~$\Pp_t$
of~$\P^\so(\La_{\oe})$ in~$\Gg_t$
is~$\SO(1,1)(k_{\E_\lprime})\times\Gg^\so$, which is normalized by the image
of~$s_t$. Since image of~$s_t$ normalizes~$\rho_{\M}^\so\chi_t$, the
induced representation~$\Ind_{\Pp_t}^{\Gg_t}\rho_{\M}^\so\chi_t$
decomposes into two pieces of equal dimension and the argument is
exactly as in previous cases, with~$\bar T_t^2=1$. Thus,
letting~$T_t^*=T_t$ be the image of~$\bar T_t$, again we have an invertible
element~$T_t^*\in\Hh(G,\l_\P^*)$ satisfying a quadratic relation
$$
(T^*_t-1)(T^*_t+1)=0,
$$
and~$q_t=q_{\so}^0=1$.

This ends the first case, so we move on to the second.
\begin{enumerate}
\setcounter{enumi}{1}
\item Suppose that~$\G_{\E_\lprime}$ is an orthogonal group and~$\dim_{\E_\lprime}\W_1=1$.
\end{enumerate}
As in case~(i) 
above, there are four possible situations. The details are almost identical to 
those in case~(i) so we omit them.

(a) Suppose first that~$\Gg_t^\so$ is non-trivial and that~$p=p_t$ 
normalizes~$\rho_\M^\so$. In this case~$s_t\in\P^\so(\MM_{t,\oe})$ and the 
argument proceeds exactly as in case~(i)(a) to give~$T_t^*\in\Hh(G,\l_\P^*)$ 
as required. 

(b) Similarly, if~$\Gg_t^\so$ is non-trivial and~$p_t\ne p$ 
normalizes~$\rho_\M^\so$, we can replace~$s_t$ by~$p_\La s_t$ to get the 
same conclusion.

(c) Now suppose~$\Gg_t^\so$ is non-trivial and~$p_t$ does not normalize~$\rho_\M^\so$ 
(in which case~$p_t\ne p$, since~$p$ normalizes~$\l_\M^\so$). In this 
case,~$p_\La=pp_t$ does not normalize~$\rho_\M^\so\chi_t$, 
and we can copy the argument in case~(i)(c) to obtain~$T_t^*\in\Hh(G,\l_\P^*)$ such that~$(T_t^*)^2=1$. 

(d) Finally suppose~$\Gg_t^\so$ is trivial, in which 
case~$\P^\so(\MM_{t,\oe})/\P_1(\MM_{t,\oe})\simeq
\SO(1,1)(k_{\F})\times\Gg_{1-t}^\so\times\Gg^\so
\simeq\P^\so(\La_{\oe})/\P_1(\La_{\oe})$. The argument is now exactly as in 
case~(i)(d).

%%%%%%%%%%%%%%%%%%%%%%%%%%%%%%%%%%%%%%%%%%%%%%%%%%%%%%%%%%%%%%%%%%%%%%%%
\subsection{}%\label{SS.nonsplithecke2}

We continue in the situation of the previous section. In all cases, we have 
two elements~$T^*_t\in\Hh(\G,\l_\P^*)$, supported on~$\J_\P^*s_t\J_\P^*$, 
which satisfy quadratic relations of the required form. The same proof as 
that of~\cite[Th\'eor\`eme~1.11]{BB2} now shows that~$\Hh(\G,\l_\P^*)$ is 
a convolution algebra on~$(\mathsf W,\{s_0,s_1\})$, where~$\mathsf W$ is 
the infinite dihedral group generated by~$s_0,s_1$.

Finally, we must see that the Hecke algebra~$\Hh(\G,\l_\P)$ has the
same form. For this, we revisit the argument of~\cite[Lemma~3.9]{M3},
which was used in deducing that~$(\J_\P,\l_\P)$ is a cover. (In fact,
we will be repeating the argument in some of the cases above.) We note
that we are in a particularly simple situation here,
as~$\J_\P/\J^*_\P$ is a product of cyclic groups of order~$2$. 

Put~$\J^*_\M=\J^*_\P\cap\M$ and~$\l_\M^*=\l_\P^*|_{\J^*_\M}$. Then,
since the difference between~$\J^*_\M$ and~$\J_\M$ is only in the
blocks~$\V^i$ with~$i<l$, the element~$s_t$ normalizes each
irreducible constituent of~$\Ind_{\J_\M^*}^{\J_\M}\l_\M^*$.  

We choose a chain of normal subgroups
\[
\J^*_\P=\K_0\subset\K_1\subset\cdots\subset\K_r=\J_\P,
\]
such that each quotient~$\K_i/\K_{i-1}$ is cyclic of order~$2$. We will 
prove, inductively on~$i$, that, for each irreducible constituent~$\l_i$ 
of~$\Ind_{\J_\P^*}^{\K_i}\l_\P^*$, there is a support-preserving 
Hecke algebra isomorphism
\[
\Hh(\G,\l_\P^*)\simeq \Hh(\G,\l_i).
\]
The case~$i=0$ is vacuous so suppose~$i\ge 1$. 

If~$\l_i|_{\K_{i-1}}$ is reducible
then~$\l_i\simeq\Ind_{\K_{i-1}}^{\K_i}\l_{i-1}$, for some irreducible
constituent~$\l_{i-1}$ of~$\Ind_{\J_\P^*}^{\K_{i-1}}\l_\P^*$. Then,
by~\cite[(4.1.3)]{BK}, we have a support-preserving isomorphism
\[
\Hh(\G,\l_{i-1})\simeq \Hh(\G,\l_i),
\]
and the claim follows by the inductive hypothesis. 

Otherwise,~$\l_{i-1}:=\l_i|_{\K_{i-1}}$ is irreducible
and~$\Ind_{\K_{i-1}}^{\K_i}\l_{i-1}$ has two irreducible
components~$\l_i=\l_i^{(1)}$ and~$\l_i^{(2)}$, which are not equivalent. 
Note that~$(\K_i,\l_i^{(j)})$ is a cover
of~$(\K_i\cap\M,\l_i^{(j)}|_{\K_i\cap\M})$, for~$j=1,2$,
by~\cite[Lemma~3.9]{M3}. We denote by~$T_t^i$ the image of~$T_t^*$
under the support-preserving isomorphism 
\[
\Hh(\G,\l_\P^*)\simeq \Hh(\G,\l_{i-1}) \simeq \Hh(\G,\Ind_{\K_{i-1}}^{\K_i}\l_{i-1})
\]
given by the inductive hypothesis and~\cite[(4.1.3)]{BK}.  We can also
identify each~$\Hh(G,\l_i^{(j)})$ as a subalgebra
of~$\Hh(\G,\Ind_{\K_{i-1}}^{\K_i}\l_{i-1})$, canonically
since~$\Ind_{\K_{i-1}}^{\K_i}\l_{i-1}$ is multiplicity free. Then,
since~$s_t$ normalizes the restrictions~$\l_i^{(j)}|_{\K_i\cap\M}$,
it follows that~$T_t^i=T^{(1)}_t+T^{(2)}_t$,
with~$T^{(j)}_t\in\Hh(G,\l^{(j)}_i)$ satisfying the same relation
as~$T^*_t$. Thus we get a support-preserving
isomorphism~$\Hh(\G,\l_\P^*)\simeq\Hh(\G,\l^{(1)}_i)=\Hh(\G,\l_i)$.
 
In particular, taking~$\l_r=\l_\P$, we deduce that~$\Hh(G,\l_\P)$,
isomorphic to~$\Hh(\G,\l_\P^*)$, as required.

%%%%%%%%%%%%%%%%%%%%%%%%%%%%%%%%%%%%%%%%%%%%%%%%%%%%%%%%%%%%%%%%%%%%%%%%
\subsection{}
This completes the proof of Theorem~\ref{thm:maximal}. Note also that 
the computation of the parameters~$q_i$ then comes down to computing 
the quadratic character~$\chi_0$ of~\S\ref{SS.nonsplithecke} and the 
parameters in the two finite Hecke 
algebras~$\Hh(\Gg_0,\rho_\M^\so\chi_0)$ and~$\Hh(\Gg_1,\rho_\M^\so)$. As mentioned above, these 
parameters can be computed using work of Lusztig~\cite{L}. Examples can 
be found in the work of Kutzko--Morris~\cite{KM} on level zero types for 
the Siegel Levi%, and in~\cite{LS} on level zero types in general
; note 
that, for level zero representations, the~$\b$-extensions are just 
trivial representations so the character~$\chi_0$ is trivial. For positive 
level representations the situation is much more subtle: see for example 
the work of Blondel~\cite{Bl}, which completely describes what happens 
when~$\L=\G$ and~$\V^i=\W_{-1}\oplus\W_1$, for some~$i\in\I_0$%; 
%and~\cite{BHS} which explores the situation when~$\G$ is a symplectic 
%group and $\M\simeq\GL_n(F)\times\Sp_4(\F)$
.

%%%%%%%%%%%%%%%%%%%%%%%%%%%%%%%%%%%%%%%%%%%%%%%%%%%%%%%%%%%%%%%%%%%%%%%%
\subsection{} 
We finish with some remarks on the Hecke algebra of a cover in the general case of a non-maximal Levi subgroup. Firstly, one interesting case is now resolved: if~$\M\simeq\GL_r(\F)^s\times\Sp_{2N}(\F)$ is a Levi subgroup of~$\Sp_{2(N+rs)}(\F)$ and~$\l_\M$ takes the form~$\til\l^{\otimes s}\otimes\l_0$, with~$\til\l$ self-dual cuspidal, then Blondel~\cite{Bl0} has given a description of the Hecke algebra, contingent on a suitable description of the Hecke algebra in the case~$s=1$ (which was already known when~$N=0$). Given Theorem~\ref{thm:maximal}, Blondel's result can now be used in full generality. 

It seems likely that the methods of~\cite{Bl0} could equally well be applied to other classical groups. However, it is not clear to the authors whether the methods used here and in~\cite{Bl0} could together be pushed to allow a description of the Hecke algebra in a completely general case.

%\bibliography{biblioSS}
\def\cprime{$'$}
\providecommand{\bysame}{\leavevmode ---\ }
\providecommand{\og}{``}
\providecommand{\fg}{''}
\providecommand{\smfandname}{\&}
\providecommand{\smfedsname}{\'eds.}
\providecommand{\smfedname}{\'ed.}
\providecommand{\smfmastersthesisname}{M\'emoire}
\providecommand{\smfphdthesisname}{Th\`ese}

\end{document}